\documentclass[11pt,a4paper]{amsart}
\pdfoutput=1
\usepackage{amssymb}
\usepackage{amsmath}
\usepackage{mathtools}
\usepackage{amsthm}
\usepackage{amsfonts}
\usepackage{bbm} 
\usepackage{tikz-cd} 
\usepackage{todonotes} 
\usepackage{graphicx} 
\usepackage[toc,page]{appendix} 
\usepackage{kpfonts} 
\usepackage[mathscr]{euscript} 
\definecolor{winered}{rgb}{0.5,0,0}

\usepackage[margin=2.5cm]{geometry} 
\usepackage{subfiles} 
 
\usepackage[style = alphabetic]{biblatex}
\renewbibmacro{in:}{}
\usepackage[citecolor = winered, linkcolor = winered, menucolor = winered, colorlinks = true]{hyperref}

\usepackage[foot]{amsaddr}

\newtheoremstyle{theoremdd}
{\topsep}{\topsep}{\upshape}{0pt}{\bfseries}{.}{ }{\thmname{#1}\thmnumber{ #2}\thmnote{ (#3)}}

\theoremstyle{definition}
\newtheorem{Th}{Theorem}[section]
\newtheorem{Lemma}[Th]{Lemma}
\newtheorem{Cor}[Th]{Corollary}
\newtheorem{Prop}[Th]{Proposition}
\newtheorem{Def}[Th]{Definition}

\newtheorem{Rem}[Th]{Remark}

\newtheorem{Ex}[Th]{Example}

\newcommand{\cat}{\mathsf}

\newcommand{\co}{\text{co}}

\newcommand{\dep}{\text{dep}}
\newcommand{\bhd}{\blacktriangleleft}
\newcommand{\rch}{\, \cat{rch} \,}
\newcommand{\join}{\vee}
\newcommand{\meet}{\wedge}
\newcommand{\Con}{\text{Con}}

\renewcommand{\and}{\wedge}
\newcommand{\orr}{\vee}
\newcommand{\rdp}{\cat{rdp}}

\newcommand{\BL}{\cat{BL}}
\newcommand{\thin}{\text{th}}
\newcommand{\en}{\text{en}}
\newcommand{\op}{\text{op}}

\makeatletter
\newtheorem*{rep@theorem}{\rep@title}
\newcommand{\newreptheorem}[2]{%
\newenvironment{rep#1}[1]{%
 \def\rep@title{#2 \ref{##1}}%
 \begin{rep@theorem}}%
 {\end{rep@theorem}}}
\makeatother
\newreptheorem{Prop}{Proposition}
\newreptheorem{Th}{Theorem}
\newreptheorem{Cor}{Corollary}

\usetikzlibrary{calc}
\usetikzlibrary{decorations.pathmorphing}

\tikzset{curve/.style={settings={#1},to path={(\tikztostart)
    .. controls ($(\tikztostart)!\pv{pos}!(\tikztotarget)!\pv{height}!270:(\tikztotarget)$)
    and ($(\tikztostart)!1-\pv{pos}!(\tikztotarget)!\pv{height}!270:(\tikztotarget)$)
    .. (\tikztotarget)\tikztonodes}},
    settings/.code={\tikzset{quiver/.cd,#1}
        \def\pv##1{\pgfkeysvalueof{/tikz/quiver/##1}}},
    quiver/.cd,pos/.initial=0.35,height/.initial=0}

\tikzset{tail reversed/.code={\pgfsetarrowsstart{tikzcd to}}}
\tikzset{2tail/.code={\pgfsetarrowsstart{Implies[reversed]}}}
\tikzset{2tail reversed/.code={\pgfsetarrowsstart{Implies}}}

\title{A Mathematical Model of Package Management Systems}

\author{Gershom Bazerman $^1$}
\address{$^1$Arista Networks}
\address{$^2$CUNY CityTech}
\address{$^3$Hyperreal Enterprises}
\author{Emilio Minichiello $^2$}
\author{Raymond Puzio $^3$}

\begin{document}

\maketitle

\setcounter{tocdepth}{1}
\begin{abstract}
This paper brings mathematical tools to bear on the study of package dependencies in software systems. We introduce structures known as Dependency Structures with Choice (DSC) that provide a mathematical account of such dependencies, inspired by the definition of general event structures in the study of concurrency. We equip DSCs with a particular notion of morphism and show that the category of DSCs is isomorphic to the category of antimatroids. We perform a similar analysis with Winskel's general event structures \cite{winskel1999event}, showing that that the category of DSCs is isomorphic to a certain subcategory of event structures without conflict. We study the exactness properties of these isomorphic categories and show that they are finitely complete with finite coproducts but do not have all coequalizers. Further, we construct a contravariant functor from a category of DSCs equipped with a certain subclass of morphisms to the category of finite distributive lattices, making use of a simple finite characterization of the Bruns-Lakser completion. Finally, we introduce a formal account of versions of packages and introduce a mathematical account of package version-bound policies.
\end{abstract}

\tableofcontents

\section{Introduction}

Package repositories and package management systems are pervasive in modern software. Such repositories can consist of binary packages in Linux distributions (such as Debian or Arch), or of source packages for use in the course of developing software and managing the libraries on which it depends (such as the npm package repository for the JavaScript language or the Hackage package repository for the Haskell language).

This paper works towards the development of a basic toolkit for considering package dependencies in software systems. Specifically we develop a mathematical model of package management systems which we call \textbf{Dependency Structures with Choice} or \textbf{DSC}s for short. A preDSC consists of a finite set $E$, whose elements we think of as packages, and a function $\dep: E \to P(P(E))$, where $P(E)$ denotes the power set of $E$. We think of this function as a dependency structure on $E$. To every package $e$ we use $\dep$ to associate a collection of ``possible dependency sets'' $D^e \in \dep(e)$. This models how packages can be run using various different sets of other packages. The inspiration for this mathematical model comes from general event structures, which were introduced for the purpose of studying concurrent computation \cite{Winskel1980EventsIC}. DSCs are then a subclass preDSCs satisfying four natural axioms that any package management system should satisfy (Definition \ref{def DSC}). To every preDSC $(E, \dep)$ we can associate a poset $\rdp(E)$, known as its \textbf{reachable dependency poset}, whose elements are those subsets of $E$ that are reachable (Definition \ref{def reachable}). One can think of a subset $A \subseteq E$ of packages as being reachable if the elements of $A$ can be ordered into a sequence of package installations such that each new package installed depends only on already installed packages. It turns out that reachable dependency posets of preDSCs inherit more structure, namely they are finite join-semilattices, which implies that they are finite lattices. This structure gives us an order-theoretic way of describing and analyzing the dependency structure inherent to a preDSC.

We define a notion of morphism for preDSCs in such a way that $\rdp$ becomes a functor from the category of preDSCs to the category of finite join-semilattices. If a preDSC $(E, \dep)$ is furthermore a DSC, then $\rdp(E)$ is in particular a finite diamond-free semimodular lattice (Definition \ref{def diamond free semimodular}). These diamond-free semimodular lattices are intimately related to antimatroids, and this led us to discover a correspondence between DSCs and antimatroids.

Antimatroids have been discovered and rediscovered for many purposes in lattice theory, graph theory, and combinatorics since the 1940s \cite{monjardet1985use}. An antimatroid consists of a set $E$ and a collection of subsets $\mathcal{F}$, whose elements are called feasible sets, satisfying the axioms of Definition \ref{def antimatroid}. Antimatroids have several applications in computer science, such as in \cite{glasserman1994monotone} and \cite{merchant2016ot}. We provide another application by showing that if $(E, \dep)$ is a DSC, then the pair $(E, \rdp(E))$ is an antimatroid. This is the content of Proposition \ref{prop rdp of a dsc is an antimatroid}. Since the literature on antimatroids does not provide a standard definition of what a morphism of antimatroids should be, we take morphisms of antimatroids to be those functions that pull feasible sets back to feasible sets. This natural choice defines a category of antimatroids. One of the main results of this paper is Theorem \ref{th equiv DSCs antimatroids}, which gives an isomorphism between the category of DSCs and the category of antimatroids. Combining this theorem with a result of Czedli \cite{Czedli_2014} then shows that any finite diamond-free semimodular lattice is isomorphic to the reachable dependency poset of a DSC.

While every DSC $(E, \dep)$ gives a finite diamond-free semimodular lattice $\rdp(E)$, it is not always the case that $\rdp(E)$ will be a distributive lattice. Distributive lattices are well studied objects of order theory, and we precisely characterize the subclass of DSCs whose reachable dependency poset is distributive. These turn out to be what we call Dependency Structures with No Choice, or those DSCs $(E, \dep)$ where $\dep(e)$ is a singleton for every $e \in E$.

We establish an isomorphism of categories between irredundant preDSCs (Definition \ref{def irredundant}) and those general event structures without conflict, which we call \textbf{conflict-free event structures} or \textbf{CFES}s. When we restrict this isomorphism to DSCs, we obtain an appropriate subcategory of CFESs in Theorem \ref{th equiv DSC and DES}, which we call \textbf{dependency event structures}. This leads also to Proposition \ref{prop family of configs}, which says that every family of configurations of a CFES can be produced by a dependency event structure. As a corollary it shows that given a preDSC $(E, \dep)$ there exists a DSC $(\widetilde{E}, \widetilde{\dep})$ and an isomorphism $\rdp(E) \cong \rdp(\widetilde{E})$.

We also prove exactness properties for the category of DSCs, equivalently the category of antimatroids and dependency event structures. In particular we prove that this category has all finite limits, but not all finite colimits. In particular, while the category of DSCs has finite coproducts, not all coequalizers exist. This can be seen in contrast to \cite{heunen2018category}, which characterizes exactness properties of the category of matroids.

Given a DSC $(E, \dep)$, its reachable dependency poset $\rdp(E)$ will in general not be a distributive lattice. One can complete a lattice using the Bruns-Lakser completion (or injective envelope) \cite{bruns1970injective}. When applied to the reachable dependency poset of a DSC, the Bruns-Lakser completion defines a Merkle tree, which gives a complete description of the execution traces of the DSC. We give an easily computed description of the Bruns-Lakser completion for finite lattices in Proposition \ref{prop bruns lakser as downsets of join irreducibles} (implicit in the work of \cite{gerhke2013}) which is of independent interest. 

Lastly, we axiomatize what it means for a package to be an updated version of another package in a DSC, characterizing it as an idempotent monad on $\rdp(E)$. We extend this to an idempotent monad on the Bruns-Lakser completion of $\rdp(E)$ as well. 

The paper is organized as follows. In Section \ref{section DSCs} we define DSCs, the reachable dependency poset construction, morphisms of DSCs, and show that with these definitions, $\rdp$ defines a contravariant functor from the category of DSCs to the category of finite join-semilattices. We also show that by restricting to a smaller class of morphisms, $\rdp$ extends to a functor to the category of finite lattices. In Section \ref{section antimatroids and dscs} we prove that the category of DSCs and the category of antimatroids are isomorphic. From this we prove that a lattice is diamond-free semimodular if and only if it can be written as the reachable dependency poset of a DSC. In Section \ref{section dsncs}, we study a subclass of DSCs where every element has a unique possible dependency set and show that it is isomorphic to the opposite of the category of finite posets. In section \ref{section event structures} we relate DSCs to event structures, which are a widely studied model of concurrent computation. While this section is independent of the rest of the paper, the connection was an initial starting point for this work, and suggests several avenues for future research. In Section \ref{section category of DSCs} we study the exactness properties of the category of DSCs. In Section \ref{section bruns-lakser}, we study the Bruns-Lakser completion of reachable dependency posets, showing that they produce Merkle trees. In Section \ref{section version parametrization} we define version parametrization for DSCs. Finally in Section \ref{section conclusion} we discuss future avenues of research.

\subsection{Relation to Other Work}
In some ways, this paper is an update to and expansion of \cite{bazerman2020topological}. The paper \cite{bazerman2021semantics} also studies DSCs, but with a different motivation and focus, it also defines morphisms of DSCs to be certain kinds of relations more general than functions. While useful, we chose to restrict the notion of morphism (Definition \ref{def morphism of DSCs}) in order to obtain the functor (\ref{eq rdp functor}). This paper can be read completely independently of \cite{bazerman2020topological} and \cite{bazerman2021semantics}.

Our definition of preDSC (Definition \ref{def preDSC}) is exactly the definition of repository in \cite[Definition 2]{dicosmo} without any notion of conflict, though our interpretation of package management systems using preDSCs is different than theirs. 

\subsection{Notation}
In what follows we use the following notation. 
\begin{enumerate}
    \item If $(E, \dep)$ is a preDSC, and $e \in E$, we let $D^e$ denote a possible dependency set of $e$,
    \item Given a set function $f: A \to B$, let $f_*: P(A) \to P(B)$ denote the image morphism, namely $f_*(S) = \{ f(s) \in B \, : \, s \in S \}$, and let $f^*: P(B) \to P(A)$ denote the preimage morphism, namely $f^*(T) = \{ a \in A \, : \, f(a) \in T \}$, for sets $S$ and $T$.
    \item We will often denote singleton sets $\{a \}$ without their braces, especially in expressions like $A \setminus \{a \}$ or $A \cup \{a \}$, we will write this as $A \setminus a$ or $A \cup a$,
    \item We often denote categories with the type face $\cat{C}$, and the set of its objects by $\cat{C}_0$. 
\end{enumerate}

\section*{Acknowledgements}
The authors would like to thank the anonoymous reviewers and the participants of the NYC Category Theory seminar for their helpful comments and suggestions.

\section{Dependency Structures with Choice} \label{section DSCs}

In this section we motivate and define the main mathematical objects of this paper, Dependency Structures with Choice. To motivate our definition, let us consider the following \texttt{package.json} file used to describe JavaScript packages in the npm package repository.

\begin{Ex} \label{ex json package}
The name and version properties of the below \texttt{package.json} file serve to uniquely identify the package within a given repository. The description and license are intended, for the most part, for human consumption, and we need not consider them here. The ``main'' property is used to indicate the entry point of the package to the build system, but we are not concerned with the actual build process at this moment, so it need not be considered either. The central thing to understand is then the ``dependencies'' property. Here, it is given as a morphismping of package names to package version specifications. So the \texttt{react} package is required to already be available at precisely version 2.4.0, and similarly for \texttt{webpack}. On the other hand, a caret precedes the version specification for \texttt{redis}. In the syntax of these files, this indicates that any version of the package with major version 4 and minor version $\geq 3$ is acceptable, i.e. it specifies not just a single version, but a range. 

 \begin{verbatim} 
{
  "name": "leftpad", 
  "version": "5.9.2", 
  "description": "Provides left padding",
  "main": "index.js",
  "license": "MIT",
  "dependencies": {
      "react": "2.4.0",
      "webpack": "0.1.3",
      "redis": "^4.3.0"
   }
}
\end{verbatim} 

Supposing we knew for a fact that the matching versions were 4.3.0, 4.4.0 and 4.5.0, then the meaning of this dependency would be the disjunctive clause: \texttt{redis-4.3.0 $\vee$ redis-4.4.0 $\vee$ redis-4.5.0}. Further, since each package gives a disjunction of versions, but the dependency field as a whole specifies a conjunction of packages, then the entirety of the key metadata for a package may be regarded as a single formula:
\begin{equation} \label{eq json formula}
    \tt{react-2.4.0} \wedge \tt{webpack-0.1.3} \wedge \tt{(redis-4.3.0 \vee redis-4.4.0 \vee redis-4.5.0)}.
\end{equation}

Metadata of package repositories other than npm may be treated similarly.
\end{Ex}

If $A$ is a set, let $P(A)$ denote its power set. There is an interesting structure on the double power set $P(P(A))$ of $A$. Suppose that $A$ is a finite set, then an element $D$ of $P(P(A))$ is of the form $D = \{ D_1, \dots, D_n \}$ where each $D_i = \{ d_1^i, d_2^i \dots, d_{m_i}^i \}$ is itself a finite subset of $A$. We can then interpret $D$ as a logical statement of the form
\begin{equation*}
    (d_1^1 \wedge d_2^1 \wedge \dots d_{m_1}^1) \vee (d_1^2 \wedge d_2^2 \wedge \dots d_{m_2}^2) \vee \dots \vee (d_1^n \wedge d_2^n \wedge \dots \wedge d_{m_n}^n).
\end{equation*}

So we can represent the formula (\ref{eq json formula}) from Example \ref{ex json package} as an element of the double power set of the packages mentioned above
\begin{equation*}
\begin{aligned}
D & = \{\{  \texttt{react-2.4.0, webpack-0.1.3, redis-4.3.0} \}, \\
& \{ \texttt{react-2.4.0, webpack-0.1.3, redis-4.4.0} \}, \\
& \{ \texttt{react-2.4.0, webpack-0.1.3, redis-4.5.0} \} \}.
\end{aligned}
\end{equation*}
The above example and observation motivate the following definition.

\begin{Def} \label{def preDSC}
A \textbf{pre-dependency structure with choice} or \textbf{preDSC} consists of a finite set $E$ and a function $\dep: E \to P(P(E))$, which we call its \textbf{dependency function}. If $e \in E$, then we call an element $D^e \in \dep(e)$, a \textbf{possible dependency set} or \textbf{depset} for short.
\end{Def}

Since we will use preDSCs to model package management systems, we wish to understand what possible states our systems can reach.

\begin{Def}
Given a preDSC $(E, \dep)$, let $X,Y \subseteq E$. We define the \textbf{reach} relation $\cat{ rch } \subseteq P(E) \times P(E)$, as follows. We write $X \, \cat{ rch } \, Y$ if
\begin{enumerate}
    \item $X \subseteq Y$, and 
    \item for every $y \in Y$, there exists a depset $D^y$ such that $D^y \subseteq X$.
\end{enumerate}
\end{Def}

Note that $\rch$ is an antisymmetric relation, i.e. if $X \rch Y$ and $Y \rch X$, then $X = Y$.

\begin{Def} \label{def reachable}
Given a preDSC $(E, \dep)$, we say that a subset $X \subseteq E$ is \textbf{reachable} if there exists a finite sequence $X_0, X_1, \dots, X_n$ of subsets of $E$ such that
$$\varnothing \rch X_0 \rch X_1 \rch \dots \rch X_n \rch X.$$
In other words, if we let $\rch^*$ denote the transitive closure of the $\rch$ relation, then $X \subseteq E$ is reachable if $\varnothing \rch^* X$.
\end{Def}

\begin{Rem}
The intuition one might have for $\cat{rch}$ is that $X \rch Y$ if everything in $Y$ can depend on things in $X$. With regards to package management, one can imagine that a set of packages is reachable if, starting from the empty state, a sequence of package installs can eventually produce that precise set of installed packages.
\end{Rem}

\begin{Lemma} \label{lem rch closed under additivity}
If $(E, \dep)$ is a preDSC, $A \rch A'$, and $B \rch B'$, then $(A \cup B) \rch (A' \cup B')$.
\end{Lemma}

\begin{proof}
If $a' \in A'$ and there exists $D^{a'} \subseteq A$ with $D^{a'} \in \dep(a')$, then $D^{a'} \subseteq A \cup B$, and similarly for all $b' \in B'$. Thus $(A \cup B) \rch (A' \cup B')$.
\end{proof}

\begin{Def} \label{def rdp}
Given a preDSC $(E, \dep)$, let $\cat{rdp}(E)$ denote the subset of $P(E)$ consisting of the reachable subsets. This inherits a partial order from the partial order $\subseteq$ on $P(E)$. We call $(\cat{rdp}(E), \subseteq)$ the \textbf{reachable dependency poset} of $E$.
\end{Def}

We now recall the following definitions. In this paper we will only consider finite posets.

\begin{Def}
Let $(P, \leq)$ be a finite poset, and $a,b \in P$.
\begin{itemize}
    \item  We say an element $z \in P$ is the \textbf{join} of $a$ and $b$ if it is the lowest upper bound or supremum of $a$ and $b$, i.e. $z \geq a$, $z \geq b$ and if $w \geq a$, $w \geq b$, then $w \geq z$. If the join of $a$ and $b$ exists, then it is unique and we denote it by $a \orr b$. A poset with binary joins is called a \textbf{join-semilattice}. A morphism of join-semilattices is an order preserving function that preserves joins.
    \item For a subset $S \subseteq P$ of a finite join-semilattice, we let $\bigvee S$ denote the supremum of all the elements of $S$. If $S = P$, then we denote $\bigvee P = \top$, and call this the \textbf{top element} of $P$.
    \item Similarly we define the \textbf{meet} of $a$ and $b$ to be the greatest lower bound or infimum of $a$ and $b$, and we denote it by $a \and b$. A poset with binary meets is called a \textbf{meet-semilattice}. A morphism of meet-semilattices is an order preserving function that preserves meets.
    \item For a subset $S \subseteq P$ of a finite meet-semilattice, let $\bigwedge S$ denote the infimum of all the elements of a finite subset $S$. If $S = P$, then we denote $\bigwedge S = \bot$ and call this the \textbf{bottom element} of $P$.
    \item A poset which is both a join-semilattice and a meet-semilattice is called a \textbf{lattice}. A morphism of lattices is an order preserving function that preserves both joins and meets.
\end{itemize}
Let $\cat{FinPos}$ denote the category of finite posets with order preserving morphisms, and similarly denote $\cat{FinMSLatt}, \cat{FinJSLatt}, \cat{FinLatt}$ for the categories of finite meet-semilattices, finite join-semilattices, and finite lattices respectively.
\end{Def}

\begin{Lemma} \label{lem reachable sets closed under union}
If $(E, \dep)$ is a preDSC, and $X,Y \subseteq E$ are reachable subsets, then $X \cup Y$ is a reachable subset, and is the join of $X$ and $Y$ in $\rdp(E)$.
\end{Lemma}

\begin{proof}
If $X$ and $Y$ are reachable, then there exist sequences $\{ X_i \}$ and $\{ Y_i \}$ such that
$$\varnothing \rch X_1 \rch X_2 \rch \dots \rch X_n \rch X$$
$$\varnothing \rch Y_1 \rch Y_2 \rch \dots \rch Y_m \rch Y$$
Suppose WLOG that $n \geq m$, then by Lemma \ref{lem rch closed under additivity} we have
$$\varnothing \rch X_1 \cup Y_1 \rch \dots \rch X_m \cup Y_m \rch X_{m+1} \cup Y_m \rch \dots \rch X_n \cup Y_m \rch X \cup Y,$$
thus $X \cup Y$ is reachable. It is easy to see that $X \cup Y$ must then be the join of $X$ and $Y$ in $\rdp(E)$.
\end{proof}

\begin{Cor} \label{cor rdp is a join semilattice}
Given a preDSC $(E, \dep)$, its reachable dependency poset $\rdp(E)$ is a finite join-semilattice.
\end{Cor}
 
Let $(E, \dep)$ and $(E', \dep')$ be preDSCs, and let $f: E \to E'$ be a function between their underlying sets. Then the preimage morphism $f^*: P(E') \to P(E)$ can be restricted to the subset of reachable subsets, $f^*: \rdp(E') \to P(E)$. We wish to consider the class of functions $f$ such that $f^*$ factors through $\rdp(E)$, which is equivalent to asking that $f^*$ preserves reachable subsets.

\begin{Def} \label{def morphism of DSCs}
Let $(E, \dep)$ and $(E', \dep')$ be preDSCs and suppose $f: E \to E'$ is a function between their underlying sets. We say that $f$ is a \textbf{morphism} of preDSCs if for every reachable subset $X \subseteq E'$, its preimage $f^*(X) \subseteq E$ is reachable. Let $\cat{PreDSC}$ denote the category of preDSCs.
\end{Def}

With this definition, the reachable dependency poset construction extends to a functor
\begin{equation} \label{eq rdp functor}
\rdp: \cat{PreDSC}^{\op} \to \cat{FinJSLatt},
\end{equation}
because by Lemma \ref{lem reachable sets closed under union} joins in $\rdp(E)$ are given by unions and given a map $f: (E, \dep) \to (E', \dep')$ the preimage map $f^* : \rdp(E') \to \rdp(E)$ preserves unions.

As a first approximation, preDSCs are a useful definition for modelling package management systems. However, it allows for dependency behavior that is unrealistic. For example, if $E = \{e \}$ and $\dep(e) = \{ \{ e \} \}$, then $e$ must depend on itself. We will now restrict our attention to certain subclasses of preDSCs that better approximate real-life package management systems.

If a preDSC $(E, \dep)$ is to model package management systems, then for a package $e \in E$, the depsets of $e$ should be minimal in some sense. For instance, if a package $e$ can be installed using only packages in a subset $D^e \subseteq E$, then it can certainly be installed using packages in $D^e \cup X$ for any subset $X \subseteq E$\footnote{This is a consequence of the fact that we are not attempting to model possible conflicts of package dependencies. Adding this feature to our model is future work.}.

\begin{Def} \label{def irredundant}
We say that a preDSC $(E, \dep)$ is \textbf{irredundant} if for every $e \in E$, and every pair $D_0^e, D_1^e \in \dep(e)$, if $D_0^e \subseteq D_1^e$, then $D_0^e = D_1^e$.
\end{Def}

If $(E, \dep)$ is an irredundant preDSC and $e \in E$, then the depsets $D^e \in \dep(e)$ model the smallest possible sets of other packages necessary to install $e$.

\begin{Def}
We say that a preDSC $(E, \dep)$ is \textbf{acyclic} if for every $e \in E$, and $D^e \in \dep(e)$, $e \notin D^e$.
\end{Def}

The acyclicity condition rules out the possibility that a package can depend on itself. However it does not rule out the possibility that a package $e$ must depend on another package $e'$ which itself must depend on $e$.

\begin{Def} \label{def can must relations}
Given $e, e' \in E$, if there exists some $D_0^e \in \dep(e)$ such that $e' \in D_0^e$, then we say that $e$ \textbf{can depend} on $e'$, which we denote by $e \leftarrow e'$. If $e' \in D^e$ for all $D^e \in \dep(e)$, then we say that $e$ \textbf{must depend} on $e'$, which we denote by $e \leftarrow_m e'$. Let $\leftarrow^*$ and $\leftarrow_m^*$ denote the transitive closure of the relations $\leftarrow$ and $\leftarrow_m$ respectively. We say that $e$ \textbf{must eventually depend} on $e'$ if $e \leftarrow_m^* e'$ and that $e$ \textbf{can eventually depend} on $e'$ if $e \leftarrow^* e'$.
\end{Def}

We wish to rule out the possibility that $e$ must eventually depend on $e$, for every $e \in E$\footnote{Note that this does not rule out the possibility that a package can eventually depend on itself.}.

\begin{Def}
Given a preDSC $(E, \dep)$, we call a subset $X \subseteq E$ \textbf{complete}, if for every $e \in X$, there exists a depset $D^e \in \dep(e)$ such that $D^e \subseteq X$. We say that a preDSC has \textbf{complete depsets} if for every $e \in E$, each depset $D^e \in \dep(e)$ is complete.
\end{Def}

\begin{Ex}
If $E = \{ a, b, c \}$ and 
\begin{equation*}
    \begin{aligned}
    \dep(a) &= \{ \{ b, c \} \} \\
    \dep(b) &= \{ \{a\}, \{c \} \} \\
    \dep(c) & = \{ \varnothing \}
    \end{aligned}
\end{equation*}
then the subset $\{b, c \}$ is complete, but $\{ a \}$ and $\{ b \}$ are not.
\end{Ex}

\begin{Lemma} \label{lem reachable implies complete}
Given a preDSC $(E, \dep)$, if a subset $X \subseteq E$ is reachable, then it is complete. Furthermore $X$ is complete if and only if $X \rch X$. 
\end{Lemma}

Finally we require that the dependency function actually assigns depsets to packages.

\begin{Def}
We say that a preDSC $(E, \dep)$ is \textbf{nonsingular} if for every $e \in E$, the set $\dep(e)$ is nonempty.  
\end{Def}

A nonsingular preDSC can of course still have the property that there is a package $e$ such that $\varnothing \in \dep(e)$. Combined with irredundancy, this implies that $e$ must depend on nothing, equivalently $\dep(e) = \{ \varnothing \}.$

\begin{Def} \label{def DSC}
We say a preDSC $(E, \dep)$ is a \textbf{dependency structure with choice} or a \textbf{DSC} if it is irredundant, acyclic, nonsingular and has complete depsets. Let $\cat{DSC}$ denote the full subcategory of $\cat{PreDSC}$ on the DSCs.
\end{Def}

These conditions are sufficient to prove that packages must not eventually depend on themselves.

\begin{Lemma}
Given a DSC $(E, \dep)$, it is not true that $e \leftarrow_m^* e$ for any $e \in E$.
\end{Lemma}

\begin{proof}
Suppose that $e \leftarrow_m^* e$. Then there exists a finite sequence $e_1, e_2, \dots, e_n$ of packages such that $e \leftarrow_m e_1$, $e_i \leftarrow_m e_{i + 1}$ for $1 \leq i \leq n-1$, and $e_n \leftarrow_m e$. Thus $e_1 \in D^e$ for every depset $D^e$. Since $D^e$ is complete by assumption, there exists a depset $D^{e_1}_0 \subseteq D^e$. But then $e_2 \in D^e$, since $e_1 \leftarrow_m e_2$. Continuing this argument, $e_i \in D^e$ implies there exists a $D^{e_i}_0$ such that $D^{e_i}_0 \subseteq D^e$, so $e_{i+1} \in D^e$. Thus $e \in D^e$, which is a contradiction.
\end{proof}

We list some properties of the two relations $\leftarrow$ and $\leftarrow_m$ for a DSC. We write $e \leftarrow \varnothing$ if $\varnothing\in \dep(e)$ and $e \leftarrow_m \varnothing$ if $\{ \varnothing \} = \dep(e)$.

\begin{Lemma} \label{lem props of relations on a DSC}
Given a DSC $(E, \dep)$ and $e, e', e'' \in E$
\begin{enumerate}
    \item if $e \leftarrow_m e'$ then $e \leftarrow e'$,
    \item if $e \leftarrow e'$ and $e' \leftarrow_m e''$, then $e \leftarrow e''$,
    \item $e \leftarrow \varnothing$ if and only if $e \leftarrow_m \varnothing$.
\end{enumerate}
\end{Lemma}

\begin{Def} \label{def covering relation on a poset}
Given a poset $(P, \leq)$, and $x, y \in P$, we write $x < y$ if $x \leq y$ and $x \neq y$. We say that $y$ \textbf{covers} $x$, and write $x \prec y$ if $x < y$ and there exists no $z \in P$ such that $x < z < y$. The Hasse diagram of $P$ is the directed graph whose nodes are the elements of $P$ and there is an edge from $x$ to $y$ if $x \prec y$. 
\end{Def}

\begin{Rem}
It is customary to drop the direction on the edges from the notation of the Hasse diagram and think of all of the edges as directed upwards. We follow this convention in this paper.
\end{Rem}

\begin{Ex}\label{ex DSC a depends on b or c}
Let $E = \{ a, b, c \}$, and $\dep(a) = \{ \{b \}, \{c \} \}$, namely $a$ can depend on $b$ or $c$, and $\dep(b) = \dep(c) = \varnothing$. Notice that this is a DSC. Consider its reachable dependency poset $\cat{rdp}(E)$. The Hasse diagram of this poset is given as follows. In our diagram, we write a subset like $\{a, b, c \}$ as $abc$.
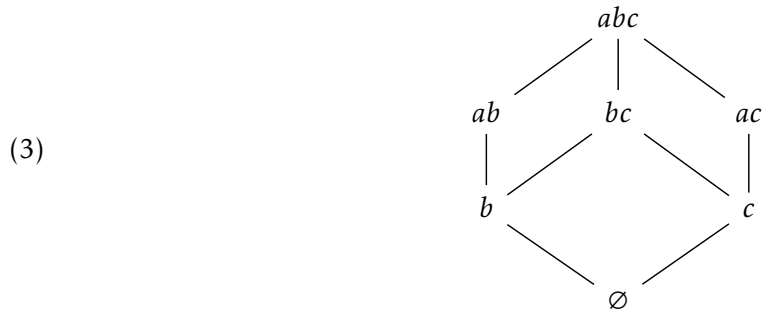
\begin{equation}
\begin{tikzcd}
	& abc \\
	ab & bc & ac \\
	b && c \\
	& \varnothing
	\arrow[no head, from=1-2, to=2-2]
	\arrow[no head, from=3-1, to=2-2]
	\arrow[no head, from=2-2, to=3-3]
	\arrow[no head, from=4-2, to=3-1]
	\arrow[no head, from=4-2, to=3-3]
	\arrow[no head, from=2-1, to=3-1]
	\arrow[no head, from=2-3, to=3-3]
	\arrow[no head, from=2-1, to=1-2]
	\arrow[no head, from=1-2, to=2-3]
\end{tikzcd}
\end{equation}
For shorthand we will refer to this DSC as $a.b \join c$, to mean that $a$ depends on $b$ or $c$ which in turn do not depend on anything.
\end{Ex}

\begin{Ex} \label{ex DSC a depends on b and c}
Let us compare Example \ref{ex DSC a depends on b or c} with a DSC $E = \{a, b, c \}$ where now $\dep(a) = \{\{b,c\} \}$, and $\dep(b) = \dep(c) = \varnothing$, namely $a$ depends on $b$ \textit{and} $c$. In this case, the Hasse diagram of $\rdp(E)$ is 
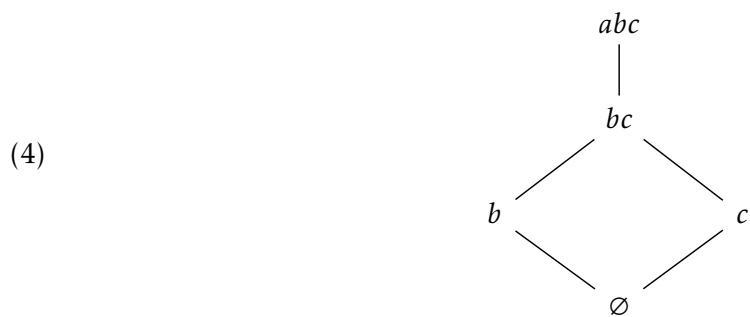
\begin{equation}
\begin{tikzcd}
	& abc \\
	& bc \\
	b && c \\
	& \varnothing
	\arrow[no head, from=1-2, to=2-2]
	\arrow[no head, from=3-1, to=2-2]
	\arrow[no head, from=2-2, to=3-3]
	\arrow[no head, from=4-2, to=3-1]
	\arrow[no head, from=4-2, to=3-3]
\end{tikzcd}
\end{equation}
For shorthand we denote this DSC by $a.b \meet c$.
\end{Ex}

\begin{Ex}
Consider $E = \{ a, b, c \}$ and $\dep(a) = \{b \}, \, \dep(b) = \varnothing, \, \dep(c) = \{ b\}$. Its reachable dependency poset is given by the Hasse diagram:
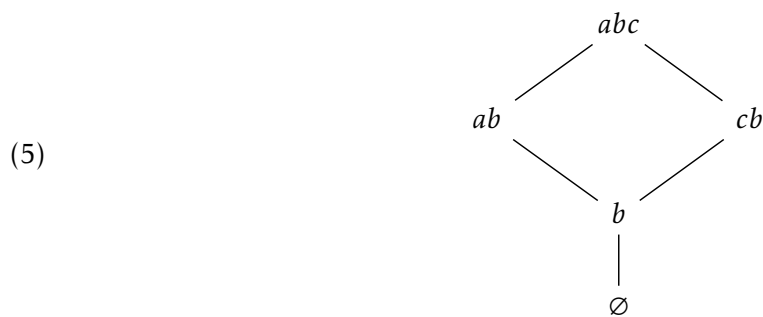
\begin{equation}
    \begin{tikzcd}
	& abc \\
	ab && cb \\
	& b \\
	& \varnothing
	\arrow[no head, from=3-2, to=4-2]
	\arrow[no head, from=2-1, to=3-2]
	\arrow[no head, from=2-3, to=3-2]
	\arrow[no head, from=1-2, to=2-1]
	\arrow[no head, from=1-2, to=2-3]
\end{tikzcd}
\end{equation}
For shorthand we will denote this DSC by $a.b,c.b$, to mean that $a$ depends on $b$ and $c$ depends on $b$.
\end{Ex}

\begin{Rem}
Notice that the intersection of two reachable subsets might not be reachable. Using the same DSC $a. b \join c$ from Example \ref{ex DSC a depends on b or c} we see that while $\{a,b \}$ and $\{a, c \}$ are reachable, their intersection $\{a \}$ is not.
\end{Rem}

Now we will prove that reachability and completeness are equivalent conditions on a subset of a DSC. Let $E$ be a DSC and $X \subseteq E$. If $e,e' \in X$, then we write $e \leftarrow_{X,m} e'$ if for every depset $D^e$ such that $D^e \subseteq X$, $e' \in D^e$. We say that $e$ \textbf{must depend on $e'$ within $X$}. If $X = E$, then $e \leftarrow_{X,m} e'$ if and only if $e \leftarrow_m e'$.

\begin{Lemma} \label{lem properties of lhd X relation}
For any nonempty subset $X \subseteq E$, $\leftarrow_{X,m}$ is a transitive relation. If $X$ is a nonempty complete subset, then $\leftarrow_{X,m}$ is also irreflexive.
\end{Lemma}

\begin{proof}
Suppose that $e \leftarrow_{X,m} e'$ and $e' \leftarrow_{X,m} e''$. Now suppose that $D^{e''}$ is a depset and $D^{e''} \subseteq X$. Then $e' \in D^{e''}$, since $e' \leftarrow_{X,m} e''$. But $D^{e''}$ is complete, since $E$ is a DSC, so there exists a depset $D^{e'}$ such that $D^{e'} \subseteq D^{e''}$. But in that case $D^{e'} \subseteq X$, so $e \in D^{e'}$ since $e \leftarrow_{X,m} e'$, which implies that $e \in D^{e''}$. Thus $e \leftarrow_{X,m} e''$.

Now if $X$ is a nonempty complete subset of $E$, and $e \leftarrow_{X,m} e$, then since $e \in X$, there exists a depset $D^e_0$ such that $D^e_0 \subseteq X$. But then $e \in D^e_0$, which is not possible since $(E, \dep)$ is a DSC. Thus $\leftarrow_{X,m}$ is irreflexive.
\end{proof}

\begin{Lemma} \label{lem can remove elements from complete subsets} 
Suppose $E$ is a DSC, and $X \subseteq E$ is a nonempty complete subset. Then there exists an $e \in X$ such that $(X \setminus e)$ is complete.
\end{Lemma}

\begin{proof}
Suppose not, then $(X \setminus e)$ is not complete for any $e \in X$. Since $X$ is nonempty, choose $e_0 \in X$. Then $(X \setminus e_0)$ is not complete. Thus there exists some $e_1 \in (X \setminus e_0)$ such that none of its depsets are contained in $(X \setminus e_0)$. However since $X$ is complete, and $e_1 \in X$, there exists a depset $D^{e_1} \subseteq X$. Since $D^{e_1} \nsubseteq (X \setminus e_0)$, this implies that $e_0 \in D^{e_1}$. Furthermore, this is true for any depset of $e_1$ that is a subset of $X$. Therefore $e_1 \leftarrow_{X,m} e_0$.

Now $(X \setminus e_1)$ is also not complete, so by the same argument there exists some $e_2$ such that $e_2 \leftarrow_{X,m} e_1$. If $|X| = n$, and we repeat this process $n$ times, we end up with a sequence
$$e_n \leftarrow_{X,m} e_{n-1} \leftarrow_{X,m} \dots \leftarrow_{X,m} e_1 \leftarrow_{X,m} e_0$$
such that either there is a cycle, i.e. $e_i = e_j$ for some $0 \leq i < j \leq n$, or the above sequence enumerates all of $X$. 

By Lemma \ref{lem properties of lhd X relation}, $\leftarrow_{X,m}$ is transitive, so if there is a cycle in the above sequence, we have $e_k \leftarrow_{X,m} e_k$ for some $0 \leq k \leq n$. However, also by Lemma \ref{lem properties of lhd X relation}, $\leftarrow_{X,m}$ is irreflexive, so we have a contradiction. 

Now suppose the above sequence enumerates all of the elements of $X$. Since $(X \setminus e_n )$ is not complete by assumption, by the same argument as above, there exists some $e_k$ such that $e_k \leftarrow_{X,m} e_n$. Therefore one obtains a cycle. However, by the previous paragraph, this cannot happen, so we have a contradiction. Therefore, there must exist some $e \in X$ such that $(X \setminus e )$ is complete.
\end{proof}

\begin{Prop} \label{prop reachability equiv to complete}
Let $E$ be a DSC, then a subset $X \subseteq E$ is reachable if and only if it is complete.
\end{Prop}

\begin{proof}
$(\Rightarrow)$ This is Lemma \ref{lem reachable implies complete}.

$(\Leftarrow)$ Suppose that $X$ is a complete subset. If $X$ is empty, then it is reachable, so assume that it is nonempty. By Lemma \ref{lem can remove elements from complete subsets}, there exists an $e_0 \in X$ such that $(X \setminus e_0 )$ is complete. Thus if $e \in (X \setminus e_0 )$, then there exists a depset $D^e \subseteq (X \setminus e_0)$. Since $X$ is complete and $e_0 \in X$, there exists a depset $D^{e_0} \subseteq X$. Since $E$ is a DSC, $e_0 \notin D^{e_0}$, thus $D^{e_0} \subseteq (X \setminus e_0 )$, thus $(X \setminus  e_0 ) \rch X$. Since $X$ is finite, say with cardinality $|X| = n$, we can repeat the above argument until we have exhausted all of $X$, obtaining a sequence
\begin{equation*}
    \varnothing \rch (X \setminus \{e_0, \dots, e_{n-1} \}) \rch \dots \rch (X \setminus e_0) \rch X.
\end{equation*}
Thus $X$ is reachable.
\end{proof}

\begin{Rem}
For the rest of the paper, we will not distinguish between reachable and complete subsets of a DSC, and will mostly refer to them as complete subsets.
\end{Rem}

Let us now analyze the structure of $\rdp(E)$ for a DSC $(E, \dep)$ more deeply.

\begin{Def} \label{def e-minimal complete subsets}
Let $(E, \dep)$ be a DSC, $e \in E$, and $X \subseteq E$. Say $X$ is an \textbf{$e$-minimal complete} subset if $X$ is complete, $e \in X$ and if for every complete subset $Y \subseteq X$ such that $e \in Y$, it follows that $X = Y$. Let $\text{Min}(e)$ denote the set of $e$-minimal complete subsets of $E$.
\end{Def}

\begin{Lemma} \label{lem complete subsets contain e-minimal complete subsets}
Suppose $(E, \dep)$ is a DSC, $e \in E$ and $X \subseteq E$ is a complete subset. If $e \in X$, then there exists an $e$-minimal complete subset $M_e \subseteq E$ such that $M_e \subseteq X$.
\end{Lemma}

\begin{Prop} \label{prop characterization of e-minimal complete subsets of a DSC}
Given a DSC $(E, \dep)$ and $e \in E$, a subset $X \subseteq E$ is an $e$-minimal complete subset if and only if it is of the form $D^e \cup e$, with $D^e \in \dep(e)$.
\end{Prop}

\begin{proof}
$(\Leftarrow)$ First note that if $D^e \in \dep(e)$, then $D^e \cup e$ is a complete subset. Now let us show that $D^e \cup e$ is an $e$-minimal complete subset. Suppose that $X$ is complete, $X \subseteq (D^e \cup e)$ and $e \in X$. Then there exists a depset $D^e_0$ of $e$ such that $D^e_0 \subseteq (X \setminus e)$, because $X$ is complete, and depsets of $e$ can't contain $e$. This implies that $D^e_0 \subseteq D^e$, but $(E, \dep)$ is irredundant, which implies $D^e_0 = D^e$, in which case $X = (D^e \cup e)$.

$(\Rightarrow)$ Now suppose that $X$ is an $e$-minimal complete subset. Then $e \in X$, and since $X$ is complete, there exists some depset $D^e \subseteq X$. Clearly $(D^e \cup e) \subseteq X$, and thus since $X$ is an $e$-minimal complete subset, $(D^e \cup e) = X$.
\end{proof}

\begin{Def} \label{def join irreducible}
Given a join-semilattice $L$, an element $x \in L$ is \textbf{join-irreducible} if $x \neq \bot$ and whenever $a \join b = x$ then $a = x$ or $b = x$. Let $\mathcal{J}(L)$ denote the subposet of join irreducible elements.
\end{Def}

Given a DSC $(E, \dep)$, we can characterize the join-irreducible elements of $\rdp(E)$.

\begin{Prop} \label{prop join irreducible iff e-minimal}
Given a DSC $(E, \dep)$, a nonempty and complete subset $X \subseteq E$ is join-irreducible in $\rdp(E)$ if and only if $X$ is an $e$-minimal complete subset for some $e \in E$.
\end{Prop}

\begin{proof}
$(\Leftarrow)$ Suppose that $X$ is an $e$-minimal complete subset and $X = A \cup B$, where $A,B$ are complete. Since $e \in X$, then $e \in A$ or $e \in B$ or both. Without loss of generality suppose that $e \in A$. Then by definition of $e$-minimal complete subsets, $X = A$. Thus $X$ is join-irreducible.

$(\Rightarrow)$ Suppose that $X$ is a complete subset of $E$ but not an $e$-minimal complete subset for any $e \in E$. We wish to show that $X$ is not join-irreducible. If $X$ is not $e$-minimal complete for any $e \in E$, then it is not $x$-minimal complete for any $x \in X$. Therefore, for every $x \in X$, by Lemma \ref{lem complete subsets contain e-minimal complete subsets} there exists an $x$-minimal complete subset $M_x \subseteq X$. Therefore $X = \bigcup_{x \in X} M_x$, so $X$ is a join of finitely many complete subsets. Thus $X$ is not join-irreducible.
\end{proof}

\begin{Prop} \label{prop uniqueness of join irreducible rep}
Let $(E, \dep)$ be a DSC and suppose that $X \subseteq E$ is join-irreducible in $\rdp(E)$. Then there exists a unique $e \in E$ and $D^e \in \dep(e)$ such that $E = (D^e \cup e)$.
\end{Prop}

\begin{proof}
By Propositions \ref{prop characterization of e-minimal complete subsets of a DSC} and \ref{prop join irreducible iff e-minimal} we know that $X = (D^e \cup e)$ for some $e \in E$ and $D^e \in \dep(e)$. Now suppose that $X = (D^{e'} \cup e')$ for $e' \in E$ and $D^{e'} \in \dep(e')$, where $e \neq e'$. Since $E$ is a DSC, both $D^e$ and $D^{e'}$ are complete subsets of $E$, and furthermore $D^e \subseteq X$ and $D^{e'} \subseteq X$, so $D^e \cup D^{e'} \subseteq X$. Since $(D^e \cup e) = (D^{e'} \cup e')$, it must be the case that $e \in D^{e'}$ and $e' \in D^e$, so  $D^e \cup D^{e'} = X$. But $X$ is join-irreducible, and union is the join in $\rdp$ so $X = D^e$ or $X = D^{e'}$. But $X$ contains both $e$ and $e'$, and since $(E, \dep)$ is acyclic, there is a contradiction. Thus $e = e'$. If $X = (D^e_0 \cup e) = (D^e_1 \cup e)$, then clearly $D^e_0 = X \setminus e = D^e_1$.
\end{proof}

Now let us show that $\rdp(E)$ is a lattice. This is a general fact about finite join-semilattices. Indeed, if $P$ is a finite join-semilattice with a bottom element and $X,Y \in P$, then their meet always exists and is given by
\begin{equation} \label{eqn defining meets in a join-semilattice}
    X \meet Y = \bigvee \left \{ Z \in P \, : \, Z \leq X, Z \leq Y \right \}.
\end{equation}
In the case of $\rdp(E)$, this means that the meet of two reachable subsets $X$ and $Y$ will be the union of all those reachable subsets that are a subset of $X \cap Y$.

\begin{Lemma} \label{lem characterization of meets in DSC}
Let $(E, \dep)$ be a DSC, and $A, B \subseteq E$ complete subsets. Then
$$A \meet B = \{ x \in A \cap B \, : \, \exists D^e \in \dep(e) \, \text{ such that } D^e \subseteq A \cap B \}.$$
\end{Lemma}

\begin{proof}
Let $C = \{ x \in A \cap B \, : \, \exists D^e \in \dep(e) \, \text{ such that } D^e \subseteq A \cap B \}.$ Now $C \subseteq A \cap B$. We wish to show that $C$ is complete. So suppose $x \in C$. Then there exists $D^x \in \dep(x)$ such that $D^x \subseteq A \cap B$. If $y \in D^x$, then since $D^x$ is complete, there exists a $D^y \subseteq D^x$ with $D^y \in \dep(y)$. Thus $D^y \subseteq A \cap B$, so $y \in C$. Since $y$ was arbitrary, $D^x \subseteq C$, so $C$ is complete. Thus $C \subseteq A \meet B$. We wish to prove the opposite inclusion. By (\ref{eqn defining meets in a join-semilattice}) if $x \in A \meet B$ then $x$ belongs to some set $X \subseteq A \cap B$ that is complete, so there exists a $D^x \subseteq X$, thus $x \in C$. Thus $A \meet B = C$.
\end{proof}

\begin{Def}
We say that a lattice $L$ is \textbf{distributive} if for every $x,y,z \in L$, the following identity holds:
\begin{equation} \label{eqn distributive identity}
 x \meet (y \join z) = (x \meet y) \join (x \meet z).
\end{equation}
\end{Def}

Distributive lattices enjoy a key characterization theorem due to Birkhoff \cite{birkhoff1940lattice}. Consider the following lattice, known as $M_3$ or the \textbf{diamond lattice}, which has five elements, $\bot, a, b, c, \top$, and order structure presented by the Hasse diagram:
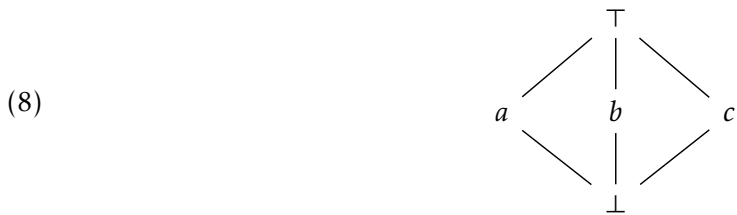
\begin{equation} \label{eqn diamond lattice}
    \begin{tikzcd}
	& \top \\
	a & b & c \\
	& \bot
	\arrow[no head, from=1-2, to=2-1]
	\arrow[no head, from=1-2, to=2-2]
	\arrow[no head, from=1-2, to=2-3]
	\arrow[no head, from=2-1, to=3-2]
	\arrow[no head, from=2-2, to=3-2]
	\arrow[no head, from=2-3, to=3-2]
\end{tikzcd}
\end{equation}
Note that this lattice is not distributive because $a \meet (b \join c) = a \meet \top = a$, but $(a \meet b) \join (a \meet c) = \bot \join \bot = \bot$.

Now let $N_5$ denote the following lattice, also known as the \textbf{pentagon lattice}, which also has five elements, $\bot, d, e, f, \top$ with the following order structure:
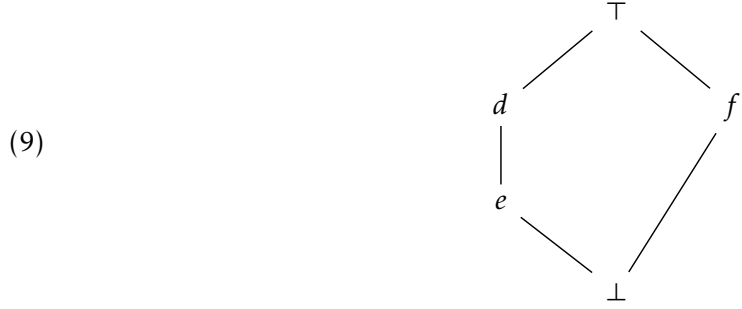
\begin{equation} \label{eqn pentagon lattice}
\begin{tikzcd} 
	& \top \\
	d && f \\
	e \\
	& \bot
	\arrow[no head, from=1-2, to=2-1]
	\arrow[no head, from=1-2, to=2-3]
	\arrow[no head, from=3-1, to=4-2]
	\arrow[no head, from=2-3, to=4-2]
	\arrow[no head, from=2-1, to=3-1]
\end{tikzcd}
\end{equation}
Note that this lattice is also not distributive.

\begin{Def}
Given a lattice $L$, a subset $S \subseteq L$ is called a \textbf{sublattice} if for all $x,y \in S$, $x \meet y \in S$ and $x \join y \in S$.
\end{Def}

\begin{Th}[{\cite[Theorem 4.10]{davey2002introduction}}]
A lattice $L$ is distributive if and only if there exists no sublattice of $L$ that is isomorphic to $M_3$ or $N_5$.
\end{Th}

Many examples of lattices are distributive, for example if $S$ is a set, then its lattice of subsets $P(S)$ is distributive. However it is typically not the case that if $E$ is a DSC, then $\rdp(E)$ will be a distributive lattice.

\begin{Ex} \label{ex rdp of a depends on b or c is not distributive}
Notice that Example \ref{ex DSC a depends on b or c} is not a distributive lattice, as
$$b \join (ab \meet ac) = b \join \varnothing = b$$
whereas
$$(b \join ab) \meet (b \join ac) = ab \meet abc = ab.$$
In other words $\{ \varnothing, b, ab, ac, abc \}$ is a sublattice of $a.b \join c$ that is isomorphic to $N_5$. In fact it contains two such sublattices, the other being $\{ \varnothing, c, ac, ab, abc \}$.
\end{Ex}

We will see in Section \ref{section antimatroids and dscs} that while the lattices in the image of $\rdp$ are not all distributive, they satisfy a weaker condition called diamond-free semimodularity, which means that they will not contain a copy of the diamond lattice $M_3$, but they could possibly contain the pentagon lattice $N_5$. Further, in Section \ref{section dsncs}, we will show that there is a subclass of DSCs whose image under $\rdp$ is precisely the set of distributive lattices.

\begin{Lemma} \label{lem morphism iff preserves reachables}
Given DSCs $(E, \dep)$ and $(E', \dep')$ and a set function $f: E \to E'$, then $f$ is a morphism of DSCs if and only if for all $e \in E$, and all $D^{f(e)} \in \dep(f(e))$, there exists a $D^e \in \text{dep}(e)$ such that $D^e \subseteq f^* \left( D^{f(e)} \cup f(e) \right)$.
\end{Lemma}

\begin{proof}
$(\Leftarrow)$ Since we are only considering DSCs, by Proposition \ref{prop reachability equiv to complete} a subset is reachable if and only if it is complete. Suppose that $A \subseteq E'$ is complete. We wish to show that $f^* (A)$ is complete. Suppose $e \in f^*(A)$. Then $f(e) \in A$, but $A$ is complete so there exists a depset $D^{f(e)} \subseteq A$. Since $f$ is a morphism of DSCs, there exists a depset $D^e \subseteq f^*(D^{f(e)} \cup f(e))$. Since $(D^{f(e)} \cup f(e)) \subseteq A$, this implies that $D^e \subseteq f^*(A)$, so $f^*(A)$ is complete.

$(\Rightarrow)$ Suppose that $f^*$ preserves complete subsets, $e \in E$ and $D^{f(e)} \in \dep(f(e))$. Then $D^{f(e)} \cup f(e)$ is complete, so $f^* \left( D^{f(e)} \cup f(e) \right)$ is complete and contains $e$, thus there exists some $D^e \in \text{dep}(e)$ such that $D^e \subseteq f^* \left( D^{f(e)} \cup f(e) \right)$. Thus $f$ is a morphism of DSCs.
\end{proof}

\begin{Rem}
By Proposition \ref{prop characterization of e-minimal complete subsets of a DSC} the above condition is equivalent to the following: $f$ is a morphism of DSCs if for every $e \in E$ and $f(e)$-minimal complete subset $M_{f(e)}$, then the preimage $f^*(M_{f(e)})$ contains an $e$-minimal complete subset.
\end{Rem}

\begin{Def} \label{def comorphism}
Given preDSCs $(E, \dep)$ and $(E', \dep')$ and a set function $f: E \to E'$, we say that $f$ is a \textbf{comorphism} if for every reachable subset $X \subseteq E$, its image $f_*(X) \subseteq E'$ is reachable. Let $\cat{PreDSC}_{\text{co}}$ and $\cat{DSC}_{\text{co}}$ denote the category of preDSCs and DSCs with comorphisms, respectively.
\end{Def}

\begin{Lemma} \label{lem co-morphism iff preserves reachables}
Let $f: E \to E'$ be a function between the underlying sets of DSCs. Then $f$ is a comorphism if and only if for all $e \in E$, and all $D^e \in \dep(e)$, there exists a $D^{f(e)} \in \dep(f(e))$ such that $D^{f(e)} \subseteq f_*(D^e \cup e)$.
\end{Lemma}

\begin{proof}
$(\Leftarrow)$ Since $(E, \dep)$ is a DSC, we can equivalently work with complete subsets. Suppose that $A \subseteq E$ is complete, we want to show that $f_*(A)$ is complete. Suppose $y \in f_*(A)$, then there exists an $e \in A$ such that $y = f(e)$. Since $A$ is complete there exists a depset $D^e \subseteq A$. By hypothesis, there exists a depset $D^{f(e)} \subseteq f_* \left( D^e \cup e \right)$. This implies that $D^{f(e)} \subseteq f_*(A)$. Thus $f_*(A)$ is complete.

$(\Rightarrow)$ Suppose that $f_*$ preserves complete subsets, $e \in E$ and $D^e \in \dep(e)$. Then $D^e \cup e$ is complete, so $f_* \left( D^e \cup e \right)$ is complete and contains $f(e)$, thus there exists a depset $D^{f(e)} \subseteq f_* \left( D^e \cup e \right)$.
\end{proof}

\begin{Cor}
If $f$ is a comorphism of preDSCs, then $f_*$ descends to an order preserving morphism of posets $f_*: \rdp(E) \to \rdp(E')$. Further, as joins are given by unions, this is a morphism of join-semilattices.
\end{Cor}

Now we wish to define a category $\cat{DSC}_{\text{bi}}$ whose objects are DSCs, and with morphisms such that $\rdp$ is a functor of type $\rdp: \cat{DSC}^{\op}_{\text{bi}} \to \cat{FinLatt}$, where $\cat{FinLatt}$ is the category whose objects are finite lattices and whose morphisms are order, join and meet-preserving functions. In other words, we wish to find a class of set functions $f$ such that $\rdp(f) = f^*$ preserves reachable subsets and is a lattice morphism, i.e. preserves joins and meets. We have already established join preservation, and now turn our attention to meets.

\begin{Lemma} \label{lem morphism comorphism}
If $f: E \to E'$ is a morphism and comorphism of DSCs, then $f^*: \rdp(E') \to \rdp(E)$ preserves meets.
\end{Lemma}

\begin{proof}
We want to show that $f^*(A \meet B) = f^*(A) \meet f^*(B)$.
Since $A \meet B \subseteq A \cap B$, we have $f^*(A \meet B) \subseteq f^*(A \cap B) = f^*(A) \cap f^*(B)$. But $f^*(A \meet B)$ is complete since $f$ is a morphism of DSCs, and $f^*(A) \meet f^*(B)$ is the largest complete subset of $f^*(A) \cap f^*(B)$, so $f^*(A \meet B) \subseteq f^*(A) \meet f^*(B)$.

Let us show that $f^*(A \meet B)$ is the largest complete subset of $f^*(A \cap B)$, as this will imply the result. Suppose that $X \subseteq f^*(A \cap B)$ is complete. Then $f_*(X) \subseteq f_* f^*(A \cap B) \subseteq A \cap B$. Since $f$ is a comorphism, $f_*(X)$ is complete, so $f_*(X) \subseteq A \meet B$. Thus $X \subseteq f^* f_*(X) \subseteq f^*(A \meet B)$. Thus $f^*(A \meet B)$ is the largest complete subset of $f^*(A \cap B)$.
\end{proof}

\begin{Def} \label{def bi-morphism of DSCs}
We say that $f: (E, \text{dep}) \to (E', \dep')$ is a \textbf{bimorphism} of DSCs if it is both a morphism and comorphism of DSCs. Let $\cat{PreDSC}_{\text{bi}}$ and $\cat{DSC}_{\text{bi}}$ denote the categories of preDSCs and DSCs with bimorphisms respectively.
\end{Def}

\begin{Cor} \label{cor bimorphisms give maps of lattices under rdp}
Given a morphism of DSCs $f: (E, \dep) \to (E', \dep')$, if $f$ is further a bimorphism, then $\rdp(f): \rdp(E') \to \rdp(E)$ is a morphism of finite lattices. Furthermore, the rdp construction extends to a functor $\rdp: \cat{DSC}_{\text{bi}}^\op \to \cat{FinLatt}$.
\end{Cor}

\begin{Rem} \label{rem morphisms of DSCs}
We have chosen morphisms of DSCs to be compatible with $f^*$ rather than $f_*$ for several reasons. One reason is that Lemma \ref{lem morphism comorphism} demonstrates that $f^*$ preserves meets when $f$ is a bimorphism, while the same does not hold for $f_*$ (this is used later in Section \ref{section bruns-lakser}). Further, this choice establishes a connection to notions of morphisms of greedoids and similar structures (including continuous functions on topological spaces), which are used in connection with antimatroids (see Remark \ref{rem morphisms of antimatroids}).
\end{Rem}

\section{Antimatroids and DSCs} \label{section antimatroids and dscs}

In this section we will establish an isomorphism between the category of antimatroids and the category of DSCs. Further, we will use this to characterize the essential image of the functor $\rdp: \cat{DSC}^\op \to \cat{FinJSLatt}$ as finite diamond-free semimodular lattices.

\subsection{Antimatroids}

Antimatroids are a mathematical construction notable for their frequent rediscovery, as chronicled in \cite{monjardet1985use}. Our work here is another instance of that rediscovery. First studied by Dilworth in 1940 \cite{dil1940}, these arise in characterizing greedy algorithms \cite{boyd1990algorithmic}, combinatorics of convex spaces \cite{Czedli_2014}, and artificial intelligence planning \cite{parmar2003some}, among other applications. 

\begin{Def} \label{def antimatroid}
An \textbf{antimatroid} is a pair $(E, \mathcal{F})$, where $E$ is a finite set and $\mathcal{F} \subseteq P(E)$ is a nonempty family of subsets, called the \textbf{feasible subsets} of $E$, such that:
\begin{enumerate}
    \item[(A1)] if $A, B \in \mathcal{F}$, then $A \cup B \in \mathcal{F}$,
    \item[(A2)] if $S \neq \emptyset$ and $S \in \mathcal{F}$, then there exists an element $x \in S$ such that $(S \setminus x) \in \mathcal{F}$, and
    \item[(A3)] $E = \bigcup_{A \in \mathcal{F}} A$.
\end{enumerate}
\end{Def}

\begin{Rem}
Condition (A3) in Definition \ref{def antimatroid} is sometimes omitted in the literature, and antimatroids that satisfy (A3) are called shelling structures in \cite{korte1985shelling}. In \cite{Czedli_2014} however they are just referred to as antimatroids. It makes little difference in the theory as one could always just consider the resulting antimatroid $(\bigcup_{A \in \mathcal{F}} A, \mathcal{F})$. 
\end{Rem}

\begin{Ex} \label{ex maximal antimatroid}
If $E$ is a finite set, then the pair $(E, P(E))$ is an antimatroid, which we call the \textbf{maximal antimatroid} of $E$.
\end{Ex}

\begin{Def}
Let $(E, \mathcal{F})$ and $(E', \mathcal{F}')$ be antimatroids and $f: E \to E'$ a set function such that if $A \in \mathcal{F}'$, then $f^*(A) \in \mathcal{F}$. We call $f$ a \textbf{morphism of antimatroids} and denote it by $f: (E, \mathcal{F}) \to (E', \mathcal{F}')$. Let $\cat{AntiMat}$ denote the category of antimatroids.
\end{Def}

\begin{Rem} \label{rem morphisms of antimatroids}
There are several notions of morphisms of antimatroids in the literature, with most sources not defining morphisms at all. The above definition agrees with a restriction of the definition of morphism of greedoids to antimatroids as given in \cite[Definition 4]{xiu2015categories}. Theorem \ref{th DSCs iso to antimatroids} will justify this choice of morphism of antimatroids.
\end{Rem}

\begin{Def} \label{def downset functor}
Given a poset $P$, a subset $A \subseteq P$ is called a \textbf{downset} if for every $x,y \in P$ such that $y \leq x$ and $x \in A$, we have that $y \in A$. In particular, let $(\downarrow x) = \{ y \in P \, : \, y \leq x \}$. Let $\mathcal{O}(P)$ denote the set of downsets of $P$, equipped with the partial order of inclusion of subsets.
\end{Def}

\begin{Rem}
If $P$ is a finite poset, then $\mathcal{O}(P)$ is a finite distributive lattice \cite[Section 3.4]{stanley2011enumerative}. In fact, every finite distributive lattice is isomorphic to the lattice of downsets of some (essentially unique) finite poset \cite[Theorem 3.4.1]{stanley2011enumerative}.
\end{Rem}

\begin{Rem} \label{rem downset remark}
The downsets construction extends to a functor $\mathcal{O} : \cat{FinPos}^\op \to \cat{FinDLatt}$, by defining $\mathcal{O}(f) : \mathcal{O}(P') \to \mathcal{O}(P)$ by $X \mapsto f^*(X)$, for a morphism $f: P \to P'$ of posets. Furthermore for every finite poset $P$, there exists a morphism $\downarrow_P : P \to \mathcal{O}(P)$ that sends $x$ to the downset $(\downarrow x)$.
\end{Rem}

\begin{Ex} \label{ex poset antimatroids}
If $P$ is a finite poset, then the pair $(P, \mathcal{O}(P))$ is an antimatroid. We call antimatroids of this form \textbf{poset antimatroids}. These are a well-studied class of antimatroids, and are the unique such class whose feasible sets are closed under intersections \cite{kempner2013geometry}. We will consider the connection between certain kinds of DSCs and this class of antimatroids in Section \ref{section dsncs}.
\end{Ex}

\begin{Lemma}
The construction of poset antimatroids forms a functor $D: \cat{FinPos} \to \cat{AntiMat}$.
\end{Lemma}

\begin{proof}
If $f: P \to P'$ is an order preserving morphism of finite posets, then let us show that $D(f): (P, \mathcal{O}(P)) \to (P', \mathcal{O}(P'))$ is a morphism of antimatroids. If $V$ is a downset of $P'$, we wish to show that $f^*(V)$ is a downset of $P$. Suppose that $x \in f^*(V)$ and $y \in P$ such that $y \leq x$. Then since $f$ is order-preserving, $f(y) \leq f(x)$, and since $V$ is a downset, $f(y) \in V$, thus $y \in f^*(V)$. Thus $f^*(V)$ is a downset of $P$. 
\end{proof}

\begin{Lemma} \label{lem posets equiv to poset antimatroids}
The functor $D: \cat{FinPos} \to \cat{AntiMat}$ is fully faithful.
\end{Lemma}

\begin{proof}
It is easy to see that if $P,P' \in \cat{FinPos}$, then the canonical function $D: \cat{FinPos}(P,P') \to \cat{AntiMat}((P,\mathcal{O}(P)), (P', \mathcal{O}(P')))$ is injective, since $D(f)$ is just $f$ on the underlying sets. We wish to show that it is also surjective. So suppose that $f: P \to P'$ is a function such that it pulls downsets of $P'$ back to downsets of $P$. We wish to show that it is also order preserving. Suppose that $y \leq x$ in $P$. Then $y \in (\downarrow x)$, and $f^*(\downarrow f(x))$ is a downset containing $x$, therefore $(\downarrow x) \subseteq f^*(\downarrow f(x))$. Thus $f(y) \in (\downarrow f(x))$, so $f(y) \leq f(x)$. Thus $D$ is fully faithful.
\end{proof}

\subsection{Correspondence of Antimatroids and DSCs} In this section we prove one of the main results of this paper, Theorem \ref{th equiv DSCs antimatroids}, which shows that the category of DSCs and the category of antimatroids are isomorphic.

\begin{Def}
Given a preDSC $(E, \dep)$, let $\Phi(E, \dep) = (E, \rdp(E))$ denote the set $E$ equipped with the family of reachable subsets $\rdp(E)$.
\end{Def}

\begin{Prop} \label{prop rdp of a dsc is an antimatroid}
If $(E, \dep)$ is a DSC, then $\Phi(E,\dep) = (E, \rdp(E))$ is an antimatroid. Furthermore a subset $X \subseteq E$ is complete in $(E, \dep)$ if and only if $X$ is feasible in $(E, \rdp(E))$.
\end{Prop}

\begin{proof}
By Lemma \ref{lem reachable sets closed under union}, complete subsets are closed under unions, proving condition (A1) of Definition \ref{def antimatroid}. Proposition \ref{prop reachability equiv to complete} and Lemma \ref{lem can remove elements from complete subsets} prove (A2). Since $E$ is complete, this proves (A3). 
\end{proof}

The construction of Proposition \ref{prop rdp of a dsc is an antimatroid} extends to a functor $\Phi: \cat{DSC} \to \cat{AntiMat}$, acting as the identity on morphisms. Let us find an inverse to this functor.

\begin{Def}
Let $(E, \mathcal{F})$ be an antimatroid, and $e \in E$. We say that a feasible subset $X \subseteq E$ is an \textbf{$e$-minimal feasible} subset if $e \in X$ and for all feasible subsets $Y$ such that $e \in Y$ and $Y \subseteq X$, it follows that $X = Y$. Let $\text{Min}_\mathcal{F}(e)$ denote the set of $e$-minimal feasible sets.
\end{Def}

\begin{Lemma}\label{lem feasible sets contain e-minimal feasible sets}
Suppose $(E, \mathcal{F})$ is an antimatroid and $X \in \mathcal{F}$ is a feasible set. If $X$ contains an element $e$, then there exists an $e$-minimal feasible set $M_e \subseteq X$.
\end{Lemma}

\begin{Lemma}\label{lem e-minimal complete subsets are complete}
Suppose that $(E, \mathcal{F})$ is an antimatroid, $e \in E$, and $M_e$ is an $e$-minimal feasible set. If $x \in (M_e \setminus e)$, then there exists an $x$-minimal feasible set $M_x \subseteq (M_e \setminus e)$.
\end{Lemma}

\begin{proof}
Since $x \in M_e$, by Lemma \ref{lem feasible sets contain e-minimal feasible sets} there exists an $x$-minimal feasible set $M_x \subseteq M_e$. Suppose that $M_e$ contains no feasible strict subsets that contain $x$, so that $M_x = M_e$. By (A2) there exists a $y \in M_e$ such that $M_e \setminus y$ is feasible. If $y \neq x$, then $(M_e \setminus y) \subset M_x$, so $(M_e \setminus y)$ is a feasible set containing $x$, which is a contradiction. If $y = x$, then $(M_e \setminus x)$ is feasible, but $M_e$ is an $e$-minimal feasible set, so $(M_e \setminus x) = M_e$, which is also a contradiction. Thus $M_x \subset M_e$. If $M_x$ contains $e$, then since $M_e$ is an $e$-minimal feasible set, $M_x = M_e$, which is a contradiction. Thus $M_x \subseteq (M_e \setminus e)$.
\end{proof}

Now suppose $(E, \mathcal{F})$ is an antimatroid. Consider the preDSC $(E, \dep_\mathcal{F})$, where $\dep_{\mathcal{F}}: E \to P(P(E))$ is defined by $\dep_{\mathcal{F}}(e) =  \{ (M_e \setminus e) \, : \, M_e \in \text{Min}_\mathcal{F}(e) \}$. If $(E, \dep)$ is a DSC, then using Proposition \ref{prop characterization of e-minimal complete subsets of a DSC} it is not hard to see that $(E, \dep) = \Psi \Phi(E, \dep)$.

\begin{Lemma} \label{lem from antimatroid to DSC}
Given an antimatroid $(E, \mathcal{F})$, the corresponding preDSC $\Psi(E, \mathcal{F}) = (E, \dep_{\mathcal{F}})$ is a DSC. Furthermore, a subset $X \subseteq E$ is complete in $\Psi(E, \mathcal{F})$ if and only if $X$ is a feasible set of $(E, \mathcal{F})$.
\end{Lemma}

\begin{proof}
We check the axioms of Definition \ref{def DSC}. If $X, Y \in \dep_{\mathcal{F}}(e)$, and $X \subseteq Y$, then since $X \cup e$ and $Y \cup e$ are $e$-minimal feasible sets, we have that $X = Y$, so $\Psi(E, \mathcal{F})$ is irredundant. If $X \in \dep_{\mathcal{F}}(e)$, then $X$ is of the form $(M_e \setminus e)$, therefore $e \notin X$, so $\Psi(E, \mathcal{F})$ is acyclic. By (A3), we know that every $e \in E$ must be contained in some feasible set, and therefore must be contained in some $e$-minimal feasible set by Lemma \ref{lem feasible sets contain e-minimal feasible sets}. Thus $\Psi(E, \mathcal{F})$ is nonsingular. If $x \in (M_e \setminus e)$, then by Lemma \ref{lem e-minimal complete subsets are complete}, there exists an $x$-minimal feasible set $M_x \subseteq (M_e \setminus e)$. Thus $(M_x \setminus x) \subseteq (M_e \setminus e)$. So $(M_e \setminus e)$ contains a depset of $x$. Thus $(M_e \setminus e)$ is complete, which implies that $\Psi(E, \mathcal{F})$ has complete depsets and therefore is a DSC.

Now if $X \in \mathcal{F}$ is feasible, then by Lemma \ref{lem feasible sets contain e-minimal feasible sets}, $X$ is a complete subset in $(E, \dep_\mathcal{F})$.

Conversely suppose that $X \subseteq E$ is complete. Then for each $x \in X$, there exists a depset $(M_x \setminus x) \subseteq X$. Of course this means $M_x \subseteq X$. Thus $\bigcup_{x \in X} M_x \subseteq X$ and $X \subseteq \bigcup_{x \in X} M_x$. Since unions of feasible sets are feasible, $X$ is feasible.
\end{proof}

Thus by Lemma \ref{lem from antimatroid to DSC}, $\Psi$ extends to a functor $\Psi: \cat{AntiMat} \to \cat{DSC}$, acting as the identity on morphisms. Furthermore if $(E, \mathcal{F})$ is an antimatroid, then $(E, \mathcal{F}) = \Psi \Phi(E, \mathcal{F})$. Thus we have proven the following result.

\begin{Th} \label{th equiv DSCs antimatroids}
The functors $\Phi$ and $\Psi$ above are inverse to each other, defining an isomorphism of categories
\begin{equation}
    \cat{DSC} \cong \cat{AntiMat}.
\end{equation}
\end{Th}

\subsection{Semimodular Lattices}
Here we review the definitions of modular and semimodular lattices and relate them to DSCs.

Given a lattice $L$ and elements $a,b \in L$, consider the sublattice $[a,b] = \{ x \in L \, : \, a \leq x \leq b \}$. Then there is an adjoint pair
$$ a \join (-) : [a \meet b, b] \rightleftarrows [a, a \join b]: (-) \meet b,$$
which is equivalent to saying that whenever $a \meet b \leq x \leq b$, then $x \leq (a \join x) \meet b$ and whenever $a \leq y \leq a \join b$, then $a \join (y \meet b) \leq y$.

\begin{Def}
We say that a lattice $L$ is \textbf{modular} if the above is an adjoint equivalence for all $a,b \in L$, i.e. if whenever $a \meet b \leq x \leq b$, $x = (a \join x) \meet b$ and whenever $a \leq y \leq a \join b$, it follows that $a \join (y \meet b) = y$.
\end{Def}

Now recall the pentagon lattice $N_5$ from (\ref{eqn pentagon lattice}).

\begin{Th}[{\cite[Theorem 4.10]{davey2002introduction}}]
A lattice $L$ is modular if and only if it does not contain $N_5$ as a sublattice.
\end{Th}

\begin{Def}
A lattice $L$ is said to be \textbf{upper semimodular} if whenever $a \meet b \prec a$, then $b \prec a \join b$, where $\prec$ is the cover relation (Definition \ref{def covering relation on a poset}. We say it is \textbf{lower semimodular} if whenever $b \prec a \join b$ then $a \meet b \prec a$.
\end{Def}

\begin{Rem}
If we refer to a lattice $L$ as \textbf{semimodular}, that means $L$ is upper semimodular. Note that a lattice is modular if and only if it is upper semimodular and lower semimodular.
\end{Rem}

\begin{Def}\label{def diamond free semimodular}
A lattice $L$ is a \textbf{diamond-free semimodular lattice} (also known as a \textbf{join-distributive lattice}) if it is semimodular and does not contain the diamond lattice $M_3$ as a sublattice.
\end{Def}

Let $(E, \mathcal{F})$ denote an antimatroid as in Definition \ref{def antimatroid}. Since $\mathcal{F} \subseteq P(E)$, it inherits a poset structure. Since $\mathcal{F}$ is closed under unions, $\mathcal{F}$ is actually a join-semilattice. Since it is finite and has a bottom element, one can also define meets as in (\ref{eqn defining meets in a join-semilattice}). In fact we have the following characterization:

\begin{Prop}[{\cite[Corollary 7.5.i]{Czedli_2014}}, originally {\cite[3.3]{edelman1980meet}}]
If $(E, \mathcal{F})$ is an antimatroid, then $(\mathcal{F}, \subseteq)$ is a diamond-free semimodular lattice.
\end{Prop}

In fact \cite[Corollary 7.5]{Czedli_2014} proves that every diamond-free semimodular lattice is of the form $(\mathcal{F}, \subseteq)$ for some antimatroid $(E, \mathcal{F})$. 

\begin{Cor} \label{cor rdp is diamond-free semimodular}
If $(E, \dep)$ is a DSC, then $\rdp(E)$ is a finite diamond-free semimodular lattice. Conversely, every finite diamond-free semimodular lattice is of the form $\rdp(E)$ for some DSC $(E, \dep)$.
\end{Cor}

In Section \ref{section dsncs} we will pinpoint those DSCs $E$ where $\rdp(E)$ is not only diamond-free semimodular but is also distributive.

\section{Dependency Structures with No Choice} \label{section dsncs}

In this section we turn to a special class of DSCs, namely those with a unique dependency set per event, i.e. those with no choice. We show that the essential image of $\rdp$ when restricted to DSNCs is equivalent to the opposite of the category of finite distributive lattices.

\begin{Def}
Suppose $(E, \dep)$ is a DSC. We say that $(E, \dep)$ is a \textbf{dependency structure with no choice} or DSNC if $\dep(e)$ is a singleton for every $e \in E$. We will abuse notation and write $\dep(e)$ to mean the unique dependency set it contains. Let $\cat{DSNC}$ denote the full subcategory of $\cat{DSC}$ whose objects are DSNCs.
\end{Def}

\begin{Lemma}
Let $(E, \dep)$ be a DSNC, with $e, e' \in E$. Then the following are equivalent
\begin{enumerate}
    \item $e' \in \dep(e)$,
    \item $e \leftarrow_m e'$,
    \item $e \leftarrow e'$.
\end{enumerate}
Furthermore if $e \leftarrow e'$ and $e' \leftarrow e''$, then $e \leftarrow e''$, i.e. the relation $\leftarrow$ is transitive.
\end{Lemma}

Given a DSNC $(E, \dep)$, let $\leftarrow^r$ denote the reflexive closure of the relation $\leftarrow$. Then $\leftarrow^r$ is still transitive, and it is also antisymmetric. Thus $\leftarrow^r$ defines a partial order on $(E, \dep)$. 

\begin{Lemma}
Given a DSNC $(E, \dep)$, let $R(E)$ denote the underlying set $E$ equipped with the partial order $\leftarrow^r$. Then $R$ extends to a functor $R: \cat{DSNC} \to \cat{FinPos}$.
\end{Lemma}

\begin{proof}
 If $f : (E, \dep) \to (E', \dep')$ is a morphism of DSNCs, then let $f: R(E) \to R(E')$ denote the same underlying function as $f$. We wish to show that if $a,b \in R(E')$ and $a \leftarrow^r b$ then $f(a) \leftarrow^r f(b)$. Since $f$ is a morphism of DSNCS, this implies that $b \in f^*(\dep(f(a)) \cup f(a))$. If $b \in f^*(\dep(f(a)))$, then $f(b) \in \dep(f(a))$, so $f(b) \leftarrow^r f(a)$.  If $b \in f^*(f(a))$, then $f(b) = f(a)$, so $f(b) \leftarrow^r f(a)$. 
\end{proof}

In what follows, if $(E, \dep)$ is a DSNC, then we will often denote the relation $\leftarrow^r$ on $E$ by $\leq$. 
 
\begin{Ex}
A simple example of a DSNC is $a.b$, the DSC with $E = \{ a, b \}$ and $\dep(a) = \{ \{ b \} \}$. Its corresponding poset is $R(E) = \{ b \leq a \}$.
\end{Ex}

The functor $R$ has a (strict) inverse $R^{-1}: \cat{FinPos} \to \cat{DSNC}$ defined as follows. If $P$ is a finite poset, then let $R^{-1}(P)$ denote the DSNC with same underlying set as $P$ and where $b \in \dep(a)$ if $b < a$. Morphisms of posets pullback downsets to downsets, thus by Lemma \ref{lem downset of a DSNC iff reachable}, letting $R^{-1}$ act on morphisms by the identity defines a morphism of DSNCs, thus $R^{-1}$ defines a functor which is a two-sided inverse to $R$.

\begin{Prop} \label{prop finpos and dsnc iso}
The functors $R$ and $R^{-1}$ defined above form an isomorphism of categories
\begin{equation}
    \cat{DSNC} \cong \cat{FinPos}.
\end{equation}
\end{Prop}

\begin{Lemma} \label{lem downset of a DSNC iff reachable}
Let $E$ be a DSNC, then a subset $X \subseteq E$ is complete if and only if $X$ is a downset in $R(E)$.
\end{Lemma}

\begin{proof}
$(\Leftarrow)$ Let $x \in X$, and $y \in \dep(x)$. Then $y \leq x$ in $R(E)$ by definition. But $X$ is a downset, so $y \in X$. Thus $\dep(x) \subseteq X$, so $X$ is a complete subset.

$(\Rightarrow)$ Suppose that $X \subseteq E$ is complete, $x \in X$ and $y \leq x$ in $R(E)$. Then either $y = x$ or $y \in \dep(x)$, in either case $y \in X$ since $X$ is complete. Thus $X$ is a downset in $R(E)$.
\end{proof}

\begin{Prop}
The following diagram (strictly) commutes
\begin{equation} \label{eq diagram for DSNCs}
\begin{tikzcd}
	{\cat{DSNC}^\op} & {\cat{DSC}^\op} \\
	{\cat{FinPos}^\op} \\
	{\cat{FinDLatt}} & {\cat{FinJSLatt}}
	\arrow["{\cat{rdp}}", from=1-2, to=3-2]
	\arrow[hook, from=1-1, to=1-2]
	\arrow["R^\op"', from=1-1, to=2-1]
	\arrow["{\mathcal{O}}"', from=2-1, to=3-1]
	\arrow[hook, from=3-1, to=3-2]
\end{tikzcd}
\end{equation}
\end{Prop}

\begin{proof}
By Lemma \ref{lem downset of a DSNC iff reachable}, we know that objectwise the two functors are equal. Similarly if $f: (E, \dep) \to (E', \dep')$ is a morphism of DSNCs, then $\mathcal{O}R^\op(f) = f^* = \rdp(f)$. Thus the morphisms agree as well.
\end{proof}

\begin{Cor} \label{cor dsncs give dist lattices}
If $(E, \dep)$ is a DSNC, then $\rdp(E)$ is a finite distributive lattice.
\end{Cor}

The above Corollary can actually be strengthened as follows. Suppose that $(E, \dep)$ is a DSC that is not a DSNC. In this case $\rdp(E)$ cannot be a distributive lattice, for if $X, Y \in \dep(e)$ are distinct elements, then the following will have to be a sublattice of $\rdp(E)$.
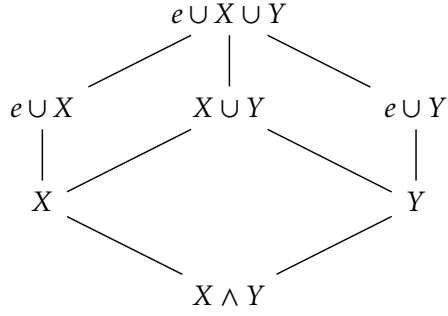
\begin{equation*}
    \begin{tikzcd}
	& {e \cup X \cup Y} \\
	{e \cup X} & {X \cup Y} & {e \cup Y} \\
	X && Y \\
	& {X \meet Y}
	\arrow[no head, from=2-1, to=3-1]
	\arrow[no head, from=1-2, to=2-1]
	\arrow[no head, from=1-2, to=2-2]
	\arrow[no head, from=1-2, to=2-3]
	\arrow[no head, from=3-1, to=4-2]
	\arrow[no head, from=3-3, to=4-2]
	\arrow[no head, from=2-2, to=3-1]
	\arrow[no head, from=2-3, to=3-3]
	\arrow[no head, from=2-2, to=3-3]
\end{tikzcd}
\end{equation*}
This sublattice of $\rdp(E)$ contains two copies of the pentagon lattice $N_5$ as discussed in Example \ref{ex rdp of a depends on b or c is not distributive}. Thus we have proved the following result.

\begin{Prop}
Given a DSC $(E, \dep)$, then $\rdp(E)$ is a finite distributive lattice if and only if $(E,\dep)$ is a DSNC.
\end{Prop}

\section{Event Structures and DSCs} \label{section event structures}

In this section, we explore a connection between DSCs and event structures, a mathematical model of concurrent computation, first defined by Winskel in his thesis \cite{Winskel1980EventsIC}. There are several different notions of event structure given in the thesis\footnote{Corresponding to elementary event structures \cite[Def. 4.1.1]{Winskel1980EventsIC} which are just posets, what are called prime event structures in \cite[Def. 1.3.4]{winskel1987event} as \cite[Def. 4.2.11]{Winskel1980EventsIC}, and \cite[Def. 3.3.1]{Winskel1980EventsIC} which are a less refined form of general event structure.} and in the concurrency literature more broadly. Here we use the following definition, which refers to what are also called general event structures, as in \cite[Section 2.2]{de2019causal}.

\begin{Def}[{\cite[Definition 1.1.1]{winskel1987event}}]
An \textbf{event structure} is a triple $(E, \Con, \vdash )$ where:
\begin{enumerate}
    \item $E$ is a set, whose elements we call \textbf{events},
    \item $\Con \subseteq P(E)$, is a nonempty collection of finite subsets called the \textbf{consistency predicate} such that if $Y \in \Con$ and $X \subseteq Y$, then $X \in \Con$, and
    \item $\vdash \, \subseteq \Con \times E$ is a relation called the \textbf{enabling relation}, which satisfies the following property: if $X \vdash e$, $X \subseteq Y$, and $Y \in \Con$, then $Y \vdash e$.
\end{enumerate}
For our purposes, we will always assume that $E$ is finite. A \textbf{conflict-free event structure} or \textbf{CFES} is an event structure $(E, \Con, \vdash)$ where $E$ is finite and $\Con = P(E)$.
\end{Def}

\begin{Rem}
Recall that relations $R \subseteq P(E) \times E$ are equivalent to functions $R' : E \to P(P(E))$. Thus the enabling relation of a CFES $(E, \Con, \vdash)$ is equivalently a function which we denote $\text{en}: E \to P(P(E))$. We will denote a CFES by $(E, \text{en})$ in what follows. 
\end{Rem}

\begin{Def}[{\cite[Definition 1.1.2]{winskel1987event}}]
A \textbf{configuration} $X$ of a CFES $(E, \en)$ is a subset $X \subseteq E$ such that for every $e \in X$, there exists a finite sequence $e_0, \dots, e_n \in C$ such that $e_n = e$ and for all $i \leq n$, $\{e_0, \dots, e_{i-1} \} \in \en(e_i)$. Let $\mathcal{F}(E)$ denote the poset of configurations of $(E, \en)$ ordered by subset inclusion, which is called the \textbf{family of configurations} of $(E, \en)$.
\end{Def}

\begin{Def}
A \textbf{morphism} of CFESs $f: (E, \en) \to (E', \en')$ is a function $f: E \to E'$ of the underlying sets, such that if $X \subseteq E'$ is a configuration, then $f^* X$ is a configuration of $E$. Let $\cat{CFES}$ denote the category of CFESs. 
\end{Def}

Clearly the family of configurations construction defines a functor $\mathcal{F}: \cat{CFES}^\op \to \cat{FinPos}.$

\begin{Def} \label{def comorphism of CFES}
A \textbf{comorphism} of CFESs $f: (E, \en) \to (E', \en')$ is a function $f: E \to E'$ of the underlying sets, such that if $C \subseteq E$ is a configuration, then $f_* C$ is a configuration of $E'$. Let $\cat{CFES}_{\text{co}}$ and $\cat{DES}_{\text{co}}$ denote the categories of CFESs and dependency event structures with comorphisms respectively. 
\end{Def}

\begin{Rem}
The comorphisms of Definitions \ref{def comorphism of CFES} are almost exactly (restrictions of) the morphisms of event structures as given in \cite[Section 2]{winskel1999event}. However there, it is also required that the underlying set function $f$ is injective. The morphisms of event structures in \cite[Section 2.1]{winskel1987event}, \cite[Section 1.4]{joyal1996bisimulation} and \cite[Definition 2.24]{graversen2019towards} are similar except that $f$ can also be a partial function. 

Note that the event structures which the morphisms in the above-cited work are defined on are different than the ones we study here -- on the one hand allowing conflict, and on the other more restricted in terms of choice (either ``stable" or ``prime").

We have have focused on morphisms that ``pull back'' configurations in our work, while much of the literature on event structures has made the opposite choice, in order to connect them with DSCs, see Remark \ref{rem morphisms of DSCs}.
\end{Rem}

\begin{Def}
Given a CFES $(E, \en)$, let $\alpha(E, \en)$ denote the preDSC given by $(E, \en_0)$, which we call the \textbf{minimal enabling preDSC} of $(E, \en)$.
\end{Def}

\begin{Lemma} \label{lem preDSC from CFES}
Given a CFES $(E, \en)$, the preDSC $\alpha(E, \en) = (E, \en_0)$ is irredundant. Furthermore, a subset $X \subseteq E$ is a configuration in $(E, \en)$ if and only if it is reachable in $\alpha(E, \en)$.
\end{Lemma}

\begin{proof}
Let $e \in E$, and suppose that $X, Y \in \en_0(e)$. If $Y \subseteq X$, then $X = Y$, by the definition of minimal enabling. Thus $(E, \en_0)$ is irredundant.
Now suppose that $X \subseteq E$ is a configuration in $(E, \en)$. Then for every $e \in X$, there exists a set $\{e_0, \dots, e_n \} \subseteq X$ such that $e_n = e$, $\varnothing \in \en_0(e_0)$ and $\{e_0, \dots, e_{i-1} \} \in \en(e_i)$ for every $1 \leq i \leq n$. Thus in $\alpha(E, \en)$ we have 
\begin{equation}
\varnothing \rch \{ e_0 \} \rch \{ e_0, e_1 \} \dots \rch \dots \{e_0, \dots, e_{n-1}, e \},
\end{equation}
because for each $e_k \in \{e_0, \dots, e_i \}$, $e_k$ is enabled by some subset of $\{e_0, \dots, e_i \}$, which means it is minimally enabled by some subset, and therefore has a depset contained in the subset.

Now suppose that $X$ is reachable in $\alpha(E, \en)$ and $e \in X$. Since $X$ is reachable, there exists sets $X_0, \dots, X_n$ with $X_n = X$ such that $\varnothing \rch X_0$ and $X_{i-1} \rch X_i$ for all $1 \leq i \leq n$. Now consider the union $Y = \bigcup_{0 \leq i \leq n -1} X_i$. Choose an ordering $\{ e_0^i, \dots, e_{m_i}^i \}$ of $X_i$ and order $Y$ as $\{e_0^0, \dots, e_{m_0}^0, e_0^1, \dots, e_{m_1}^1, \dots, e_0^{n-1}, \dots, e_{m_{n-1}}^{n-1} \}$. Relabel this sequence as $\{y_0, \dots, y_{k-1} \}$ and let $y_k = e$. Then $\{y_0, \dots, y_k \}$ forms a sequence such that $\varnothing \in \en_0(y_0)$, $y_k = e$ and $\{y_0, \dots, y_{i-1} \} \in \en(y_i)$ for all $1 \leq i \leq k$. Thus $X$ is a configuration in $(E, \en)$.
\end{proof}

By Lemma \ref{lem preDSC from CFES}, the construction $\alpha(E, \en)$ extends to a functor $\alpha: \cat{CFES} \to \cat{iPreDSC}$, with $\alpha$ acting as the identity on morphisms, where $\cat{iPreDSC}$ denotes the full subcategory of $\cat{PreDSC}$ on the irredundant preDSCs.

Given an irredundant preDSC $(E, \dep)$, let $\beta(E, \dep)$ denote the CFES $(E, \en)$ with $\en : E \to P(P(E))$ defined for $e \in E$ by
\begin{equation}
    \en(e) = \{ X \subseteq E \, : \, \exists Y \subseteq X, \, Y \in \text{dep}(e) \}.
\end{equation}
It is easy to see that $(E, \dep) = \alpha \beta(E, \dep)$, which implies by Lemma \ref{lem preDSC from CFES} that if $X \subseteq E$, then $X$ is reachable in $(E, \dep)$ if and only if it is a configuration in $\beta(E, \en)$. Thus $\beta$ extends to a functor $\beta: \cat{iPreDSC} \to \cat{CFES}$, acting as the identity on morphisms, and furthermore we have that if $(E, \en)$ is a CFES, then $\beta \alpha(E, \en) = (E, \en)$.  Thus we have proven the following.

\begin{Prop} \label{prop cfes equiv to ipredsc}
The functors $\alpha$ and $\beta$ above define an isomorphism of categories 
\begin{equation*}
\cat{CFES} \cong \cat{iPreDSC},
\end{equation*}
such that the following diagram strictly commutes
\begin{equation}
\begin{tikzcd}
	{\cat{CFES}^\op} && {\cat{iPreDSC}^\op} \\
	& {\cat{FinPos}}
	\arrow["{\alpha^\op}"', curve={height=6pt}, from=1-1, to=1-3]
	\arrow["{\mathcal{F}}"', from=1-1, to=2-2]
	\arrow["\rdp", from=1-3, to=2-2]
	\arrow["{\beta^{\op}}"', curve={height=6pt}, from=1-3, to=1-1]
\end{tikzcd}
\end{equation}
\end{Prop}

\begin{Def}
Given a CFES $(E, \text{en})$, with $e \in E$, we say that $X \subseteq E$ is a \textbf{minimal enabling set} for $e$ if $X \in \en(e)$ and for all $Y \subseteq X$ if $Y \in \en(e)$, then $X = Y$. We write $X \in \en_0(e)$ to mean that $X$ is a minimal enabling set for $e$. This defines a relation $\en_0 : E \to P(P(E))$, which we call the minimal enabling relation. We say that a CFES $(E, \en)$ is \textbf{acyclic} if for any $X \subseteq E$ such that $X \in \en_0(e)$, then $e \notin X$. We say that $(E, \en)$ is \textbf{nonsingular} if for every $e \in E$, there exists a subset $X \subseteq E$ such that $X \in \en(e)$. We say that $(E, \en)$ has \textbf{configurable minimal enabling sets} if for every $e \in E$ and every $X \in \en_0(e)$, $X$ is a configuration. We say that a CFES $(E, \en)$ is a \textbf{dependency event structure} if it is acyclic, nonsingular and has configurable minimal enabling sets.
\end{Def}

Let $\cat{DES}$ denote the full subcategory of $\cat{CFES}$ on the dependency event structures.

\begin{Lemma} \label{lem DES to DSC}
If $(E, \en)$ is a dependency event structure, then $\alpha(E, \en)$ is a DSC. Conversely if $(E', \dep)$ is DSC, then $\beta(E', \dep)$ is a dependency event structure.
\end{Lemma}

\begin{proof}
By Lemma \ref{lem preDSC from CFES}, $\alpha(E, \en)$ is irredundant. Acyclicity and nonsingularity of $(E, \en)$ translate straightforwardly to acyclicity and nonsingularity of $\alpha(E, \en)$ as a preDSC. By Lemma \ref{lem preDSC from CFES}, a subset is reachable if and only if it is a configuration, thus the minimal enabling sets being configurations implies that all depsets are reachable and therefore complete.

Conversely if $(E, \dep)$ is a DSC, then $\beta(E, \dep)$ is clearly acyclic and nonsingular, and subsets of $E$ are configurations in $\beta(E, \dep)$ if and only if they are complete in $(E, \dep)$, thus all minimal enabling sets are configurations, so $\beta(E, \dep)$ is a dependency event structure.
\end{proof}

\begin{Th} \label{th equiv DSC and DES}
The functors $\alpha$ and $\beta$ above restrict to an isomorphism of categories
\begin{equation}
    \cat{DES} \cong \cat{DSC}.
\end{equation}
\end{Th}

Combining Theorem \ref{th equiv DSC and DES} and Corollary \ref{cor rdp is diamond-free semimodular} gives the following result.

\begin{Cor}
If $(E, \en)$ is a dependency event structure, then $\mathcal{F}(E)$ is a finite diamond-free semimodular lattice.
\end{Cor}

In fact, by the following result, we have characterized the family of configurations of all conflict-free event structures.

\begin{Prop} \label{prop family of configs}
For every CFES $(E, \en)$, there exists a dependency event structure $(\widetilde{E}, \widetilde{\en})$ such that $\mathcal{F}(E) \cong \mathcal{F}(\widetilde{E})$.
\end{Prop}

\begin{proof}
Given the CFES $(E, \en)$, let $(E^1, \en^1)$ denote the CFES with $E^1 = E$ and $\en^1(e) = \{ X \in \en(e) \, : \, X \text{ is a configuration} \}$. If $X \subseteq E$ is a configuration in $(E, \en)$, then for every $e \in X$ there exists a sequence of elements $e_0, \dots, e_n \in X$ such that $\{e_0, \dots, e_{i-1} \} \in \en(e_i)$. This means that every set $\{e_0, \dots, e_{i-1} \}$ is a configuration. Thus $X$ is a configuration in $(E^1, \en^1)$. Conversely if $X$ is a configuration in $(E^1, \en^1)$, then it is certainly a configuration in $(E, \en)$. Thus $\mathcal{F}(E) = \mathcal{F}(E^1)$. Note that $(E^1, \en^1)$ is also acyclic. Indeed if $X \subseteq E^1$ is an enabling set for $e \in (E^1, \en^1)$ that contains $e$, then it is a configuration, so there exists a proper subset of $X$ that does not include $e$ that enables $e$, so $X$ is not minimal. 

Given a CFES $(E, \en)$, let $\text{Sing}(E) \subseteq E$ denote those elements of $E$ that are not enabled by any subsets, i.e. those $e \in E$ such that $\en(e) = \varnothing$. Let $(E^2, \en^2)$ be the CFES with $E^2 = E^1 \setminus \text{Sing}(E^1)$ and $\en^2$ is defined so that for $e \in E^2$, $\en^2(e) = \{ X \setminus \text{Sing}(E^1) \, : \, X \in \en^1(e) \}$. If $X \subseteq E^1$ is a configuration in $(E^1, \en^1)$, then it is still a configuration in $(E^2, \en^2)$, because configurations cannot contain singular elements. If $X \subseteq E^2$ is a configuration in $(E^2, \en^2)$, then it is clearly a configuration in $(E^1, \en^1)$. Thus there is an isomorphism $\mathcal{F}(E^2) \cong \mathcal{F}(E^1)$.

Now $(\widetilde{E}, \widetilde{\en}) \coloneqq (E^2, \en^2)$ is a dependency event structure whose family of configurations is isomorphic to $\mathcal{F}(E)$.
\end{proof}

\begin{Cor}
The family of configurations of CFESs take values in diamond-free semimodular lattices. Furthermore if $(E, \dep)$ is a preDSC, then there exists a DSC $(\widetilde{E}, \widetilde{\dep})$ and an isomorphism $\rdp(E) \cong \rdp(\widetilde{E})$.
\end{Cor}

\begin{proof}
Given a preDSC $(E, \dep)$, let $(E', \dep')$ denote its irredundant hull (Definition \ref{def irredundant hull}). Clearly there exists an isomorphism $\rdp(E) \cong \rdp(E')$. Then $\alpha(E', \dep')$ is a CFES, so by Proposition \ref{prop family of configs} there exists a dependency event structure $(\widetilde{E}, \widetilde{\en})$ such that $\mathcal{F}(E) \cong \mathcal{F}(\widetilde{E})$. Then by Lemma \ref{lem DES to DSC},  $(\widetilde{E}, \widetilde{\dep}) = \beta(\widetilde{E}, \widetilde{\en})$ is a DSC such that $\rdp(\widetilde{E}) \cong \rdp(E') \cong \rdp(E)$.
\end{proof}

There is a further well known subclass of event structures.

\begin{Def}
We say that an event structure $(E, \Con, \vdash)$ is \textbf{stable} if for all $e \in E$, if $X \vdash e$, $Y \vdash e$, and $X \cup Y \cup e \in \Con$, then $X \cap Y \vdash e$.
\end{Def}

Suppose that $(E, \en)$ is a CFES that is also stable. Since $\Con = P(E)$, being stable implies that whenever $X \in \en(e)$ and $Y \in \en(e)$, then $X \cap Y \in \en(e)$. Thus if $(E, \en)$ is furthermore a stable dependency event structure, then $\alpha(E, \en)$ is a DSC with the property that if $e \in E$ with two depsets $D_0^e, D^1_e \in \dep(e)$, then $D_0^e \cap D_1^e \in \dep(e)$. But $\alpha(E, \en)$ is irredundant, thus it must be the case that $D_0^e = D_1^e$. In other words, stable dependency event structures must have unique minimal enabling sets. Let $\cat{SDES}$ denote the full subcategory of stable dependency event structures.

\begin{Cor}
The functors $\alpha$ and $\beta$ above restrict to an isomorphism of categories
\begin{equation}
    \cat{SDES} \cong \cat{DSNC}.
\end{equation}
\end{Cor}

\section{The Category of DSCs} \label{section category of DSCs}
In this section we study exactness properties of the category of DSCs. Let $U : \cat{DSC} \to \cat{FinSet}$ denote the forgetful functor $U(E, \dep) = E$. Let us show that this functor has a left adjoint. Define the free functor $F: \cat{FinSet} \to \cat{DSC}$ by $F(S) = (S, \dep_S)$, where $\dep_S : S \to P(P(S))$ is defined by $\dep_S(s) = \{ \varnothing \}$. We call $F(S)$ the \textbf{discrete} DSC on $S$. Clearly if $f: S \to U(E, \dep)$ is a set function, then $f$ will be a morphism of DSCs $f: F(S) \to (E, \dep)$, as $\varnothing \subseteq f^*(D^{f(s)} \cup s)$ for all $s \in S$. This proves the following result.

\begin{Lemma} \label{lem DSC is a concrete category}
The above functors define an adjunction $F: \cat{FinSet} \rightleftarrows \cat{DSC}: U$. Further, the forgetful functor $U: \cat{DSC} \to \cat{FinSet}$ is faithful, making $\cat{DSC}$ into a concrete category.
\end{Lemma}

This result implies that $U$ preserves whatever limits exist in $\cat{DSC}$. Thus if a limit of a diagram of DSCs exists, it must be a DSC whose underlying set is a limit of the underlying sets in the diagram.

\begin{Prop} \label{prop dscs have initial and terminal objects}
The category $\cat{DSC}$ has initial and terminal objects.
\end{Prop}

\begin{proof}
Consider the empty set $\varnothing$ equipped with the unique function $\dep_\varnothing: \varnothing \to P(P(\varnothing))$. This is vacuously a DSC. Clearly there is a unique morphism $(\varnothing, \dep_\varnothing) \to (E, \dep)$ for any DSC. Similarly there is a unique morphism $(E, \dep) \to F(*) = (*, \dep_*)$, where $*$ denotes the singleton set.
\end{proof}

Since $U: \cat{DSC} \to \cat{FinSet}$ preserves limits, if $f: (E,\dep) \to (E', \dep')$ is a monomorphism of DSCs, then it must be a monomorphism of sets. Now $f$ is a monomorphism if and only if for every DSC $X$, the function $\cat{DSC}(X,f): \cat{DSC}(X,E) \to \cat{DSC}(X,E')$ is injective. Equivalently, we have the following commutative diagram
\begin{equation*}
\begin{tikzcd}
	{\cat{DSC}(X,E)} && {\cat{DSC}(X,E')} \\
	{\cat{FinSet}(UX,UE)} && {\cat{FinSet}(UX,UE')}
	\arrow["{\cat{DSC}(X,f)}", from=1-1, to=1-3]
	\arrow["{\cat{FinSet}(UX,Uf)}"', from=2-1, to=2-3]
	\arrow["{U_E}"', from=1-1, to=2-1]
	\arrow["{U_{E'}}", from=1-3, to=2-3]
\end{tikzcd}
\end{equation*}
and since $U$ is faithful, the vertical morphisms are injective. Thus if $Uf$ is injective, then so is $\cat{FinSet}(UX,Uf)$, and therefore so is $\cat{DSC}(X,f)$, so we have proven the following result. 

\begin{Lemma}
A morphism $f: (E, \dep) \to (E', \dep')$ of DSCs is a monomorphism if and only if $Uf$ is an injective function.
\end{Lemma}

The same reasoning as above then proves the following result.

\begin{Lemma}
If $f : E \to E'$ is a morphism of DSCs such that $U(f)$ is a surjection of sets, then $f$ is an epimorphism of DSCs.
\end{Lemma}

We now wish to show that if a morphism $f: (E,\dep) \to (E', \dep')$ of DSCs is an epimorphism, then $U(f)$ is a surjection of sets. So suppose $f: E \to E'$ is a morphism of DSCs such that $U(f)$ is not a surjection. We wish to show that it cannot be an epimorphism in the category $\cat{DSC}$. To do so, let us define a construction on DSCs that "doubles" an element.

Let $(E, \dep)$ be a DSC, and let $b \in E$. We will define a preDSC $E_{b\to b_1, b_2}$ as follows. Let the underlying set of $E_{b \to b_1, b_2}$ be defined as $(E \setminus b) \cup \{b_1, b_2 \}$. Let us define its dependency structure $\dep'$ as follows. Let $\dep'(b_1) = \dep'(b_2) = \dep(b)$. For every $e \in E \setminus b$, and $D^e \in \dep(e)$, if $b \notin D^e$, then let
\begin{equation}
    D^e_{b \to b_1, b_2} = \begin{cases}
        D^e & b \notin D^e, \\
        (D^e \setminus b) \cup \{b_1, b_2\} & b \in D^e.
    \end{cases}
\end{equation} 
This will still be a complete subset of $E_{b \to b_1, b_2}$ since $\dep'(b_1) = \dep'(b_2) = \dep(b)$ and $D^e$ was complete in $E$. Let $\dep'(e)$ contain all of the $D^e_{b \to b_1, b_2}$ indexed over the depsets $D^e$ in $E$. This defines a preDSC, and it is easy to check that it is in fact a DSC.

\begin{Prop}
Let $f: (E, \dep) \to (E', \dep')$ be a morphism of DSCs such that $U(f)$ is not a surjection. Then $f$ is not an epimorphism in $\cat{DSC}$.
\end{Prop}

\begin{proof}
If $U(f)$ is not a surjection, then there exists a $b \in E'$ such that $b \notin f(E)$. Let $E'' = E'_{b \to b_1, b_2}$ be the ``doubling" construction defined above, and consider two functions $g_1, g_2: E' \to E''$ with $g_1(e') = g_2(e') = e'$ for all $e' \neq b$, and have $g_1(b) = b_1$ and $g_2(b) = b_2$. It is easy to see from the construction that $g_1 f = g_2 f$. Since $g_1 \neq g_2$, this implies that $f$ is not an epimorphism.
\end{proof}

\begin{Cor} \label{cor epimorphism iff surjection}
A morphism $f: (E, \dep) \to (E', \dep')$ of DSCs is an epimorphism if and only if $U(f)$ is a surjection of sets.
\end{Cor}

Consider two DSCs $(E, \dep), (E', \dep')$. We can take the coproduct of their underlying sets $E + E'$, and consider the function $\dep_+: E + E' \to P(P(E + E'))$ defined by 
$$\dep_+(e) = \begin{cases}
\dep(e) & \text{if } e \in E \\
\dep'(e) & \text{if } e \in E'. \end{cases}$$
It is not hard to show that $(E + E', \dep_+)$ is a DSC.

\begin{Lemma}
Given two DSCs $(E, \dep), (E', \dep')$, the DSC $(E + E', \dep_+)$ is their categorical coproduct in the category $\cat{DSC}$.
\end{Lemma}

\begin{proof}
It is obvious that the inclusion morphisms $\text{inl} : (E, \dep) \to (E + E', \dep_+)$ and $\text{inr}: (E', \dep') \to (E + E', \dep_+)$ are morphisms of DSCs. So suppose that there are morphisms $f: (E, \dep) \to (Q, \dep^Q)$ and $g: (E', \dep') \to (Q,\dep^Q)$. We wish to show that the induced set function $( f, g ) : E + E' \to Q$ is a morphism of DSCs. Namely we need to show that for every $x \in E + E'$, and every $D^q \in \dep^Q(q)$ where $q = (f,g)(x)$, there exists a $D^x \in \dep_+(x)$ such that $D^x \subseteq (f,g)^{-1}(D^q \cup q)$. WLOG suppose that $x = e$ for some $e \in E$, then $(f,g)(x) = f(e)$. Since $f$ is a morphism of DSCs, this implies that there exists a $D^e \in \dep(e)$ such that $D^e \subseteq f^{-1}( D^q \cup q)$. Now since $f^{-1}(D^q \cup q)$ is a subset of $E$, we can also consider it as a subset of $E + E'$. In this sense it is clear that $f^{-1}(D^q \cup q) \subseteq (f,g)^{-1}(D^q \cup q)$, as the latter is equal to $f^{-1}(D^q \cup q) + g^{-1}(D^q \cup q)$. The case for $x = e'$ with $e' \in E'$ is similar. Thus $(f,g)$ is a morphism of DSCs. It is not hard to check that $(f,g)$ is the unique morphism of DSCs such that $(f,g) \circ \text{inl} = f$ and $(f,g) \circ \text{inr} = g$.
\end{proof}

This also proves that the forgetful functor $U: \cat{DSC} \to \cat{FinSet}$ preserves coproducts.

Let $(E, \dep)$ and $(E', \dep')$ be DSCs. Consider the set theoretic product $E \times E'$. We wish to put a DSC structure on this set. Let $(e, e') \in E \times E'$. Define a function $\dep_\times : E \times E' \to PP(E \times E')$ by $\dep_\times (e,e') = \{ e \times D^{e'} \cup D^{e} \times e' \cup D^e \times D^{e'} \, : \, D^e \in \dep(e), D^{e'} \in \dep'(e') \}$.  It is not hard to check that $(E \times E', \dep_\times)$ is a DSC.

\begin{Prop} \label{prop dscs have products}
Given two DSCs $(E,\dep), (E',\dep')$, the DSC $(E \times E', \dep_\times)$ is their categorical product in the category $\cat{DSC}$.
\end{Prop}

\begin{proof}
First we must show that the projection morphisms $\pi_1: (E \times E', \dep_\times) \to (E, \dep)$ and $\pi_2: (E \times E', \dep_\times) \to (E', \dep')$ are morphisms of DSCs. We will show this for $\pi_1$. If $(e,e') \in E \times E'$, then we need to show that for every $D^e \in \dep(e)$, there exists a $D^{(e,e')} \in \dep(e,e')$ such that
\begin{equation} \label{eqn projection morphism}
    D^{(e,e')} \subseteq \pi_1^{-1} \left( D^e \cup e \right).
\end{equation}
Now $\pi_1^{-1} \left( D^e \cup e \right) = D^e \times E' \cup e \times E'$. Thus for any $D^{e'} \in \dep'(e')$, we have $e \times D^{e'} \cup D^e \times e' \cup D^e \times D^{e'} \subseteq e \times E' \cup D^e \times E'$. Thus (\ref{eqn projection morphism}) holds. The case for $\pi_2$ holds similarly.

Now let us prove that $(E \times E', \dep_\times)$ satisfies the universal property of a product. Namely if we have morphisms of DSCs $f: (Q, \dep^Q) \to (E, \dep)$ and $g: (Q, \dep^Q) \to (E', \dep')$, we wish to show that the set theoretic morphism $\langle f, g \rangle : Q \to E \times E'$ is a morphism of DSCs. If $q \in Q$, $f(q) = e, g(q) = e'$ and $D^{(e,e')} \in \dep_\times(e,e')$, then there exists a $D^q \in \dep^Q(q)$ such that
\begin{equation}
    D^q \subseteq \langle f, g \rangle^{-1} \left( D^{(e,e')} \cup (e,e') \right).
\end{equation}
Now $D^{(e,e')} = D^e \times D^{e'} \cup D^e \times e' \cup e \times D^{e'}$ for some $D^e \in \dep(e), D^{e'} \in \dep'(e')$, and thus
\begin{equation} \label{eqn preimage of product}
\begin{aligned}
\langle f, g \rangle^{-1} \left( D^{(e,e')} \cup (e,e') \right) & = [f^{-1}(D^e) \cap g^{-1}(D^{e'})] \cup [f^{-1}(D^e) \cap g^{-1}(e')] \\ & \cup [f^{-1}(e) \cap g^{-1}(D^{e'})] \cup [f^{-1}(e) \cap g^{-1}(e')]
\end{aligned}
\end{equation}
Now since $f$ and $g$ are morphisms, we have $D^q \subseteq f^{-1}(D^e) \cup f^{-1}(e)$ and $D^q \subseteq g^{-1}(D^{e'}) \cup g^{-1}(e').$ Thus $D^q \subseteq (f^{-1}(D^e) \cup f^{-1}(e) ) \cap (g^{-1}(D^{e'}) \cup g^{-1}(e'))$. Now if we let 
$$f^{-1}(D^e) = X, \, g^{-1}(e') = Y, \, g^{-1}(D^{e'}) = Z, \, f^{-1}(e) = W,$$ 
then we have
\begin{equation*}
    \begin{aligned}
(f^{-1}(D^e) \cup f^{-1}(e) ) \cap (g^{-1}(D^{e'}) \cup g^{-1}(e')) & = (X \cup W) \cap (Z \cup Y) \\
& = (X \cap (Z \cup Y)) \cup (W \cap (Z \cup Y)) \\
& = (X \cap Z) \cup (X \cap Y) \cup (W \cap Z) \cup (W \cap Y).
    \end{aligned}
\end{equation*}
and this last expression is equal to
$$
[f^{-1}(D^e) \cap g^{-1}(D^{e'})] \cup [f^{-1}(D^e) \cap g^{-1}(e')] \cup [f^{-1}(e) \cap g^{-1}(D^{e'})] \cup [f^{-1}(e) \cap g^{-1}(e')].
$$
Which is precisely (\ref{eqn preimage of product}). Thus the morphism $\langle f, g \rangle$ is a morphism of DSCs.
\end{proof}

It is also easy to see that products of discrete DSCs will be discrete DSCs. In other words, the functor $F: \cat{FinSet} \to \cat{DSC}$ preserves products.

\begin{Def} \label{def irredundant hull}
Let $(E, \dep)$ be a preDSC. We say that a depset $D^e \in \dep(e)$ is a \textbf{minimal depset} if for every $D_0^e \in \dep(e)$ if $D_0^e \subseteq D^e$, then $D^e = D_0^e$. Given a preDSC $(E, \dep)$ let $(E, \widetilde{\dep})$ denote the preDSC where $\widetilde{\dep}(e)$ is the set of minimal depsets of $e$. Then $(E, \widetilde{\dep})$ is irredundant. We call $(E, \widetilde{\dep})$ the \textbf{irredundant hull} of $(E, \dep)$.
\end{Def}

Suppose that $(E, \dep)$ is a DSC, and $A \subseteq E$ is a subset. Let $\dep|_A$ denote the function $\dep|_A : A \to PP(A)$ defined so that $D^a \in \dep|_A(a)$ if there exists a $D_0^a \in \dep(a)$ such that $D^a = D_0^a \cap A$. The pair $(A, \dep|_A)$ forms a preDSC, which we call the \textbf{restriction} of $E$ to $A$, but in general it is not an irredundant preDSC.

\begin{Lemma} \label{lem subset DSC}
If $(E, \dep)$ is a DSC, and $A \subseteq E$ is a subset, then the irredundant hull of its restriction $(A, \widetilde{\dep|_A})$ is a DSC. Further, the inclusion $i: (A, \widetilde{\dep|_A}) \to (E, \dep)$ is a morphism of DSCs.
\end{Lemma}

\begin{proof}
Clearly $(A, \widetilde{\dep|_A})$ satisfies (D0) by construction. Clearly (D1) and (D2) hold, since they hold for $(E, \dep)$. Let us show that it satisfies (D3), namely that every depset is complete. Suppose that $a \in A$, $D^a \in \widetilde{\dep|_A}(a)$ and $x \in D^a$. Since $D^a = D^a_0 \cap A$ for some $D^a_0 \in \dep(a)$, this implies that $x \in A$ and $x \in D^a_0$. Since $(E, \dep)$ is a DSC, this means that there exists some $D^x \in \dep(x)$ such that $D^x \subseteq D^a_0$. Thus $D^x \cap A \subseteq D^a$. If $D^x \cap A$ is a minimal dependency set for $x$ in $(A, \dep|_A)$, then $D^x \cap A \in \widetilde{\dep|_A}(x)$. If $D^x \cap A$ is not a minimal dependency set for $x$ in $(A, \dep|_A)$, then there exists some $D^x_0 \in \dep(x)$ such that $D^x_0 \cap A$ is a minimal depset of $x$ and $D^x_0 \cap A \subseteq D^x \cap A$. Thus $D^x_0 \cap A \subseteq D^a$ and $D^x_0 \cap A \in \widetilde{\dep|_A}(x)$. Since $x$ was arbitrary $D^a$ is complete. Thus $(A, \widetilde{\dep|_A})$ satisfies (D3), and therefore is a DSC.

Let us show that the inclusion is a morphism of DSCs. Namely if $a \in A$, then for every $D^a \in \dep(a)$, there exists a $D^a_0 \in \widetilde{\dep|_A}(a)$ such that $D^a_0 \subseteq D^a \cup a$. But this is clearly true as one could take $D^a_0$ to be the minimal depset contained in $D^a \cap A$.
\end{proof}

If $(E, \dep)$ is a DSC, and $A \subseteq E$, then we call $(A, \widetilde{\dep|_A})$ the \textbf{subset DSC} of $E$ with respect to $A$. We can also say that we have equipped $A$ with the subset DSC structure.

Now suppose that there are two DSC morphisms $f,g : (E, \dep) \to (E', \dep')$. We can consider the subset $\text{Eq}(f,g) = \{e \in E \, : \, f(e) = g(e) \} \subseteq (E, \dep)$, which is the categorical equalizer in $\cat{FinSet}$. Thanks to Lemma \ref{lem subset DSC}, equipping $\text{Eq}(f,g)$ with the subset DSC structure gives a DSC for which the inclusion morphism is a morphism of DSCs. We wish to show that this is a categorical equalizer of $f$ and $g$.

\begin{Prop}\label{prop dscs have equalizers}
Given two morphisms of DSCs $f,g: (E, \dep) \to (E', \dep')$, the set $\text{Eq}(f,g)$ equipped with the subset DSC structure from $(E, \dep)$, denoted by $(\text{Eq}(f,g), \widetilde{\dep})$ is an equalizer for the two morphisms in the category $\cat{DSC}$.
\end{Prop}

\begin{proof}
Suppose there is a DSC $(Q, \dep^Q)$ and a morphism $h: (Q, \dep^Q) \to (E, \dep)$ such that $fh = gh$. We wish to show that $h$ factors uniquely through the inclusion morphism $$i: {(\text{Eq}(f,g), \widetilde{\dep}) \to (E, \dep)}.$$ We know that as set functions, $h$ factors uniquely through $i$, resulting in the following commutative diagram in $\cat{FinSet}$
\begin{equation*}
    \begin{tikzcd}
	Q \\
	{\text{Eq}(f,g)} & E & {E'}
	\arrow["h", from=1-1, to=2-2]
	\arrow["k"', dashed, from=1-1, to=2-1]
	\arrow["i"', from=2-1, to=2-2]
	\arrow["f", shift left=2, from=2-2, to=2-3]
	\arrow["g"', shift right=2, from=2-2, to=2-3]
\end{tikzcd}
\end{equation*}
We need only show that $k$ is a morphism of DSCs. Namely if $q \in Q$ we wish to show that for every $D^{k(e)} \in \widetilde{\dep}(k(e))$ there exists a $D^q \in \dep^Q(q)$ such that $D^q \subseteq k^{-1}(D^{k(e)} \cup k(e) )$. Since $h(Q) \subseteq \text{Eq}(f,g)$, $k$ is actually the corestriction of $h$, namely it is the same function with codomain restricted to $\text{Eq}(f,g)$. Thus $D^{k(e)} = D^{h(e)}$, and this is of the form $D^{h(e)}_0 \cap \text{Eq}(f,g)$ for some $D^{h(e)}_0 \in \dep(h(e))$ that makes $D^{h(e)}$ a minimal dependency set. Thus we need to show that there exists some $D^q$ such that
$$D^q \subseteq h^{-1} \left( (D^{h(e)}_0 \cap \text{Eq}(f,g)) \cup h(e) \right).$$
But the right hand side is equal to $h^{-1}(D^{h(e)}_0) \cap h^{-1}(\text{Eq}(f,g)) \cup h^{-1}(h(e))$. Since $h$ is equalized by $f$ and $g$, $h^{-1}(\text{Eq}(f,g)) = Q$. Thus we need only show that there exists a $D^q$ such that $D^q \subseteq h^{-1}(D^{h(e)}_0 \cup h(e))$. This is guaranteed since $h$ is a morphism of DSCs. Thus $k$ is a morphism of DSCs, and therefore an equalizer in $\cat{DSC}$.
\end{proof}

\begin{Cor} \label{cor dscs are finitely complete}
The category $\cat{DSC}$ is finitely complete.
\end{Cor}

\begin{proof}
The category $\cat{DSC}$ has a terminal object by Proposition \ref{prop dscs have initial and terminal objects}, binary products by Proposition \ref{prop dscs have products} and equalizers by Proposition \ref{prop dscs have equalizers}. Since every finite limit can be written using a terminal object, binary products and equalizers, $\cat{DSC}$ is finitely complete.
\end{proof}

The case for colimits of DSCs is more subtle than limits. The following example showcases this subtlety.

\begin{Ex}
Consider the DSC $a.b \join c$, and the two morphisms $a,c: * \to a.b \join c$. We wish to show that there does not exist a coequalizer of these two morphisms in the category of DSCs. Suppose there was, denote it by 
$$\begin{tikzcd}
	{*} & {a.b \vee c} & E
	\arrow["a", shift left=1, from=1-1, to=1-2]
	\arrow["c"', shift right=1, from=1-1, to=1-2]
	\arrow["f", from=1-2, to=1-3]
\end{tikzcd}$$
Now since $f$ is a coequalizer, we know that it must be an epimorphism. By Corollary \ref{cor epimorphism iff surjection}, this implies that $f$ is a surjection on the underlying sets. Since $f$ is a coequalizer, we know that its image $f(a.b \join c)$ must have cardinality $\leq 2$, since $f$ must identify $a$ and $c$. This implies that $|E| \leq 2$, where $|E|$ denotes the cardinality of $E$. Up to isomorphism, there are exactly three DSCs with cardinality $\leq 2$. They are the terminal DSC $*$, the discrete DSC on two elements $x, y$ and the DSC $u.v$. Now there exists no morphism of DSCs $a.b \join c \to x, y$, so it cannot be the coequalizer. There exists a morphism of DSCs $f: a.b \join c \to u.v$ with $f(a) = f(c) = u$ and $f(b) = v$, and since this morphism isn't constant, it doesn't factor through $*$, so $*$ cannot be the coequalizer. Thus $u.v$ is the only possible candidate for $E$. However, consider the morphism $h: a.b \join c \to r.s \join t$ with $h(a) = h(c) = r$ and $h(b) = s$. This is a morphism of DSCs, and if $u.v$ was the coequalizer, then $h$ would have to factor uniquely through $f$, namely there would have to exist a unique morphism of DSCs $k: u.v \to r. s \join t$ such that $kf = h$. This implies that $k(u) = r$ and $k(v) = s$. However this is not a morphism of DSCs. For instance $k^*(t \cup r) = u$, so there exists no depset of $u$ in $u.v$ that is a subset of $k^*(t \cup r)$. Thus $f$ does not factor uniquely through $h$, hence $u.v$ cannot be the coequalizer. This proves there cannot exist any coequalizer for the above diagram.   
\end{Ex}

\begin{Cor}
The category $\cat{DSC}$ does not have all coequalizers.
\end{Cor}

However by Proposition \ref{prop finpos and dsnc iso}, we know that the full subcategory of DSNCs is finitely complete and cocomplete, because the category of finite posets is. 

\section{The Bruns-Lakser Completion} \label{section bruns-lakser}

In this section we will consider a construction known in the order theory literature as the \textbf{Bruns-Lakser Completion} or \textbf{Bruns-Lakser Injective Envelope}. It was first defined in \cite{bruns1970injective}, and it defines an idempotent completion of meet-semilattices to distributive lattices in a way that preserves certain kinds of joins. Here we will describe this construction, apply it to the image of lattices under $\rdp$ and show how the composition of Bruns-Lakser after $\rdp$ provides an interesting interpretation of DSCs as a kind of data structure known in the computer science literature as a Merkle tree.

First we will recall the classical Bruns-Lakser completion. Later we will recast this construction as the Yoneda embedding of a meet-semilattice into its category of ``thin sheaves," as described in \cite{stubbe2005canonical}.

\begin{Def}
Let $L$ be a meet-semilattice. We say that a subset $S \subseteq L$ is a \textbf{distributive subset} if for every $x \in L$, the following joins exist and the following identity holds:
    $$ x \meet \bigvee S = \bigvee (x \meet S),$$
where $\bigvee (x \meet S)$ is the join of the subset $\{ x \meet s \, : \, s \in S \}$. We call elements of the form $\bigvee S$ for a distributive subset a \textbf{distributive join}. We say a morphism $f: L \to L'$ of meet-semilattices preserves distributive joins if whenever $S$ is a distributive subset of $L$, then $f_*(S)$ is a distributive subset of $L'$, and $f(\bigvee S) = \bigvee f_*(S)$. We call such a morphism $f$ a \textbf{distributive morphism} of meet-semilattices.
\end{Def}

\begin{Rem}
In some papers such as \cite{gerhke2013} and \cite{bruns1970injective} they refer to distributive subsets as admissible subsets. A distributive morphism of meet-semilattices is called a $\langle \meet, a\join \rangle$-morphism in \cite{gerhke2013}.
\end{Rem}

For our purposes, the condition requiring joins existing is vacuous, since $\rdp(E)$ is a lattice for any DSC $E$. However, not all subsets of $\rdp(E)$ will be distributive, and thus even though the Bruns-Lakser construction was originally defined for meet-semilattices, it will apply to our case as well.

\begin{Def}
A \textbf{distributive ideal} in a meet-semilattice $L$ is a subset $A \subseteq L$ such that:
\begin{enumerate}
    \item $A$ is a downset, (Definition \ref{def downset functor}) and
    \item if $I \subseteq A$ and $I$ is a distributive subset of $L$, then $\bigvee I \in A$.
\end{enumerate}
\end{Def}

Let $\BL(L)$ denote the subposet of the power set $P(L)$ consisting of distributive ideals ordered under set inclusion. We refer to this as the \textbf{Bruns-Lakser completion} of $L$.

\begin{Lemma} \label{lem prime downsets are distributive ideals}
If $a \in L$, then $(\downarrow a)$ is a distributive ideal.
\end{Lemma}

\begin{proof}
If $I \subseteq (\downarrow a)$ is any subset such that $\bigvee I$ exists, since $\bigvee I \leq a$ by definition, then $\bigvee I \in (\downarrow a)$.
\end{proof}

Let $b_L : L \to \cat{BL}(L)$ denote the injective order-preserving function $a \mapsto (\downarrow a)$. 

\begin{Lemma}[{\cite[Lemma 3]{bruns1970injective}}]
For a meet-semilattice $L$, $\cat{BL}(L)$ is a complete lattice, namely it has infinite joins and meets. Meets correspond to set intersections, while joins are given by:
$$ \bigvee_{i \in I} A_i = \left \{ \bigvee X \, : \, X \subseteq \bigcup_{i \in I} A_i, \; X \text{ is distributive in } L \right \},$$
where $\{ A_i \}_{i \in I}$ is a collection of elements of $\cat{BL}(L)$. Furthermore the function $b_L : L \to \cat{BL}(L)$ preserves meets and distributive joins.
\end{Lemma}

\begin{Lemma}[{\cite[Corollary 1]{bruns1970injective}}] \label{lem BL is distributive lattice}
For any meet-semilattice $L$, $\cat{BL}(L)$ is a complete infinitely distributive lattice. Furthermore, if $L$ is finite, then $\cat{BL}(L)$ is a finite distributive lattice.
\end{Lemma}

\begin{Def}\label{def distributive morphism}
Let $E,E'$ be DSCs. A \textbf{distributive morphism} of DSCs is a set function $f: E \to E'$ such that $\rdp(f): \rdp(E') \to \rdp(E)$ is a distributive morphism of meet-semilattices. Let $\cat{DSC}_{\text{dis}}$ denote the category of DSCs and distributive morphisms.\footnote{See \cite[Section 3]{gerhke2013} for more information on distributive subsets and morphisms. It is not hard to see meet semilattices with distributive morphisms form a category, which implies that $\cat{DSC}_{\text{dis}}$ also forms a category.}
\end{Def}

\begin{Lemma}
If $f: E \to E'$ is a bimorphism of DSCs and is injective, then $f$ is a distributive morphism of DSCs.
\end{Lemma}

\begin{proof}
If $f$ is injective, then $f^* f_*(X) = X$ for all $X \subseteq E$. By Lemma \ref{lem co-morphism iff preserves reachables}, if $X \subseteq E$ is reachable, then $f_*(X)$ is reachable, and $f^* f_*(X) = X$, thus $f^*$ is surjective. Since $f$ is a bimorphism, by Lemma \ref{cor bimorphisms give maps of lattices under rdp} $f^*$ also preserves meets and joins. Thus by \cite[Proposition 3.12]{gerhke2013}, $f^*$ preserves distributive subsets.
\end{proof}

The Bruns-Lakser completion extends to a functor $\cat{BL}: \cat{FinMSLatt}^\op_{\text{dis}} \to \cat{FinDLatt}$, where $\cat{FinMSLatt}_{\text{dis}}$ is the category of meet-semilattices and distributive morphisms. Indeed, if $f: L \to L'$ is a distributive morphism of meet-semilattices, then let $\BL(f) : \BL(L') \to \BL(L)$ denote the order-preserving function given by $X \mapsto f^*(X)$. Since $f$ preserves meets and joins, $\BL(f)$ is a morphism of finite distributive lattices. For our purposes we will restrict $\cat{BL}$ to the full subcategory $\cat{FinLatt}_{\text{dis}}$ of $\cat{FinMSLatt}_{\text{dis}}$ whose objects are finite lattices.

We wish to view DSCs as some kind of generalized space, in the hope of being able to bring to bear modern mathematical tools to study package management systems. The particular generalization of space we will use is known as a \textbf{frame} or \textbf{locale}\footnote{Typically the category of locales is taken to be the opposite of the category of frames.}. They are a generalization of the structure of the lattice of open subsets of a topological space. One can also think of the theory of locales as a sort of decategorified topos theory, see \cite{johnstone1982stone}. This is the viewpoint we will take here.

\begin{Def}
A \textbf{frame} $L$ is a poset with arbitrary joins, finite meets and satisfying an infinitary version of the distributive law, namely
\begin{equation*}
    x \meet \bigvee_{i \in I} y_i = \bigwedge_{i \in I} x \vee y_i,
\end{equation*}
for any set $I$. A morphism $f: L \to L'$ is a function that preserves finite meets and arbitrary joins.
\end{Def}

\begin{Rem}
Note that a finite poset is a frame if and only if it is a distributive lattice. Therefore, we will continue to use the term finite distributive lattice in what follows, though we will be working in the context of what would usually be called frame or locale theory.
\end{Rem}

\begin{Def}
Let $L$ be a poset. A \textbf{thin presheaf} on $L$ is a functor (equivalently an order-preserving function) $\varphi: L^{op} \to 2$, where $2$ denotes the poset $2 \coloneqq \{ 0 \leq 1 \}$. A morphism of thin presheaves is a natural transformation of such functors. Let $\cat{Pre}^\thin(L)$ denote the category of thin presheaves on $L$.
\end{Def}

Note that a natural transformation $h: \varphi \Rightarrow \psi$ of thin presheaves is unique if it exists. In other words, $\cat{Pre}^{\thin}(L)$ is a poset. Furthermore there is an isomorphism $F: \cat{Pre}^{\thin}(L) \to \mathcal{O}(L)$ of posets. Indeed, given an order preserving morphism $\varphi: L^{op} \to 2$, then let $F(\varphi) = \varphi^{-1}(1)$, which is a downset. Conversely, given a downset $X \subseteq L$, one can define a thin presheaf $F^{-1}(X)$ to be $1$ on $X$ and $0$ elsewhere. This isomorphism of posets extends to a natural isomorphism of functors
\begin{equation} \label{eq iso between presheaves and downsets}
\begin{tikzcd}
	{\cat{FinPos}^{\op}} & {\cat{FinDLatt}}
	\arrow[""{name=0, anchor=center, inner sep=0}, "{\cat{Pre}^{\thin}}", curve={height=-12pt}, from=1-1, to=1-2]
	\arrow[""{name=1, anchor=center, inner sep=0}, "{\mathcal{O}}"', curve={height=12pt}, from=1-1, to=1-2]
	\arrow["F \Downarrow"{description}, draw=none, from=0, to=1]
\end{tikzcd}
\end{equation}

 Given a finite poset $L$, there is a canonical monomorphism $y_L: L \hookrightarrow \cat{Pre}^\thin(L)$ given by the decategorified Yoneda embedding, sending $x \in L$ to the morphism $y_L(x): L^{op} \to 2$ defined by 
$$y_L(x)(u) = \begin{cases} 1 & \text{if } u \leq x, \\
0 & \text{else}.
\end{cases}$$
For every finite poset $L$, we obtain a commutative diagram
\begin{equation} \label{eq downsets and representables diagram}
    \begin{tikzcd}
	& L \\
	{\cat{Pre}^{\thin}(L)} && {\mathcal{O}(L)}
	\arrow["{F}", from=2-1, to=2-3]
	\arrow["{y_L}"', hook', from=1-2, to=2-1]
	\arrow["{\downarrow_L}", hook, from=1-2, to=2-3]
\end{tikzcd}
\end{equation}

\begin{Def}
Let $L$ be a meet-semilattice, and $x \in L$. We say that a subset $R \subseteq L$ is a family over $x$ if whenever $r \in R$, then $r \leq x$. A \textbf{Grothendieck pretopology} on a meet-semilattice $L$ is a function $J$ that assigns to every point $x \in L$ a set of families $J(x)$ satisfying the following conditions:
\begin{enumerate}
    \item The singleton $x \in J(x)$,
    \item if $R \in J(x)$ and $y \leq x$, then $R \meet y \in J(y)$, and
    \item if $R \in J(x)$, and for every $r_i \in R$, there exists a family $S_i \in J(r_i)$, then $\bigcup_i S_i \in J(x)$.
\end{enumerate}
We call the families $R \in J(x)$ that belong to a Grothendieck pretopology \textbf{covering families}. We call a meet-semilattice equipped with a Grothendieck pretopology a \textbf{posite}. A morphism of posites $f: L \to L'$ is an order preserving morphism that preserves meets and covering families, in the sense that if $R \in J(x)$, then $f_*(R) \in J(f(x))$.
\end{Def}

Let $L$ be a meet-semilattice, then define $J_\text{dis}$ to be the function that assigns to every $x \in L$ the set of distributive subsets $R \subseteq L$ such that $\bigvee R = x$. This defines a Grothendieck pretopology on $L$. In fact $J_\text{dis}$ generates what is known as the canonical Grothendieck topology on $L$, see \cite[Theorem 1]{stubbe2005canonical}. Note that a morphism $f: L \to L'$ of meet-semilattices is a morphism of posites (where we equip $L$ and $L'$ with $J_{\text{dis}}$) if and only if it is a distributive morphism of meet-semilattices.

\begin{Def}
Given a posite $(L, J)$, $x \in L$ and a thin presheaf $\varphi: L^{op} \to 2$, we say that $R \in J(x)$ is a \textbf{matching family} for $x$ if $\varphi(r) = 1$ for all $r \in R$. 

We say that $\varphi$ is a \textbf{thin sheaf} if for every $x \in L$, and every matching family $R$ over $x$, $\varphi(x) = 1$. Let $\cat{Sh}^\thin(L)$ denote the full subcategory of $\cat{Pre}^{\thin}(L)$ on the thin sheaves over $L$.
\end{Def}

By the natural isomorphism (\ref{eq iso between presheaves and downsets}), a thin presheaf on a finite lattice $L$ is equivalently a downset of $L$. Thus a thin sheaf is a downset of $L$ that is furthermore closed under joins of distributive subsets. In other words, a thin presheaf is a thin sheaf if and only if its corresponding downset is a distributive ideal. This observation leads to the following result.

\begin{Th}[{\cite{stubbe2005canonical}}] \label{th bruns lakser iso to thin sheaves}
Let $L$ be a meet-semilattice, considered as a posite $(L, J_{\text{dis}})$, then there is an isomorphism of distributive lattices
$$\mathsf{BL}(L) \cong \cat{Sh}^{\thin}(L),$$
where $L$ is considered as a posite $(L, J_{\text{dis}})$.
\end{Th}

\begin{Cor} \label{cor distributive pretopology is subcanonical}
Given a meet-semilattice $L$, the Grothendieck pretopology $J_\text{dis}$ on $L$ is subcanonical, namely every representable thin presheaf on $(L, J_{\text{dis}})$ is a thin sheaf.
\end{Cor}

\begin{proof}
This follows by combining Theorem \ref{th bruns lakser iso to thin sheaves}, the commutative diagram (\ref{eq downsets and representables diagram}) and Lemma \ref{lem prime downsets are distributive ideals}.
\end{proof}

From Corollary \ref{cor distributive pretopology is subcanonical}, we see that the Yoneda embedding factors through the category of thin sheaves. Thus for every meet-semilattice $L$, we obtain the following commutative diagram
\begin{equation}
    \begin{tikzcd}
	& L \\
	{\cat{Sh}^{\thin}(L)} && {\cat{BL}(L)} \\
	{\cat{Pre}^{\thin}(L)} && {\mathcal{O}(L)}
	\arrow["\cong", from=3-1, to=3-3]
	\arrow["{y_L}"', hook', from=1-2, to=2-1]
	\arrow["{b_L}", hook, from=1-2, to=2-3]
	\arrow[hook', from=2-1, to=3-1]
	\arrow[hook, from=2-3, to=3-3]
	\arrow["\cong", from=2-1, to=2-3]
\end{tikzcd}
\end{equation}

Let us delve more deeply into the Bruns-Lakser completion when restricted to finite lattices. Recall that if $L$ is a poset, then $\mathcal{J}(L)$ denotes the subposet of join-irreducible (Definition \ref{def join irreducible}) elements of $L$, ordered by subset inclusion. Thus $\mathcal{O J}(L)$ consists of the downsets (Definition \ref{def downset functor}) of $L$ which are made up only of join-irreducible elements. Given a poset $L$ let $\widehat{(-)}_L : L \to \mathcal{O} \mathcal{J}(L)$ denote the morphism of posets that sends $x \in L$ to the set $\widehat{x} = \{ y \leq x : y \in \mathcal{J}(L) \}$. If $L$ is a meet-semilattice, then $\widehat{(-)}$ preserves meets. This morphism is well known to order theorists thanks to the following celebrated theorem of Birkhoff.

\begin{Th}[{\cite{birkhoff1937rings}}] \label{th birkhoffs theorem}
If $L$ is a finite distributive lattice, then the map $\widehat{(-)}_L : L \to \mathcal{O J}(L)$ is an isomorphism.
\end{Th} 

We will show that the Bruns-Lakser completion applied to a finite lattice $L$ coincides exactly with $\mathcal{O J}(L)$. To do so, we need some preliminary results.

\begin{Lemma}[{\cite[Lemma 3.4]{gerhke2013}}] \label{lem technical lemma on dist subsets}
Given a finite lattice $L$, then $X \subseteq L$ is a distributive subset if and only if for every $y \in \mathcal{J}(L)$, if $y \leq \bigvee X$, then $y \leq x$ for some $x \in X$. 
\end{Lemma}

\begin{Rem}
Lemma \ref{lem technical lemma on dist subsets} is equivalent to \cite[Lemma 3.4]{gerhke2013} in the case that $L$ is finite by noting that finite lattices are complete, and therefore we can take $(e, L^\delta) = (1_L, L)$, the canonical extension of $L$, to be the identity. Similarly an element of a finite lattice is completely join irreducible ($x = \bigvee A$ implies $x \in A$) if and only if it is join irreducible, so $J^\infty(L) = \mathcal{J}(L)$. 
\end{Rem}

\begin{Cor} \label{cor equiv def of distr property}
Given a finite lattice $L$, a subset $A \subseteq L$ is distributive if and only if
\begin{equation}
    \widehat{\bigvee A} = \bigcup_{a \in A} \widehat{a}.
\end{equation}
\end{Cor}

\begin{Def}
Let $L$ be a poset with $x, y \in L$. We say that $x$ is a \textbf{predecessor} of $y$ if $x < y$. Let $(\Downarrow x)$ denote the subposet of predecessors of $x$.
\end{Def}

\begin{Lemma} \label{lem Ray's lemma}
Let $L$ be a finite lattice and $x \in L$, then $x = \bigvee \widehat{x}$.
\end{Lemma}

\begin{proof}
We will prove this by induction on the cardinality of $(\Downarrow x)$. For the base case, if $|\Downarrow x| = 0$, then $x = \bot$. In this case, $\widehat{\bot} = \varnothing$, and $\bigvee \varnothing = \bot$.

Now suppose that for every $y \in L$ with $|\Downarrow y| \leq n$ it follows that $y = \bigvee \widehat{y}$. Let $x \in L$ with $| \Downarrow x| \leq n + 1$. If $x$ is join-irreducible, then $x \in \widehat{x}$, and thus $x = \bigvee \widehat{x}$. Suppose that $x$ is not join-irreducible. Then there exists a subset $S \subseteq L$ with $x \notin S$ such that $x = \bigvee S$. Therefore for every $s \in S$, $s < x$. Thus $|\Downarrow s| < |\Downarrow x| \leq n + 1$. Thus by the induction hypothesis, $s = \bigvee \widehat{s}$. Therefore
\begin{equation}
    x = \bigvee S = \bigvee \left \{ \bigvee \widehat{s} : s \in S \right \} = \bigvee \bigcup_{s \in S} \widehat{s}.
\end{equation}
Now suppose that $y \in \mathcal{J}(L)$ and $y \leq s$ for some $s \in S$. Then since $s \leq x$, it follows that $y \leq x$, so $y \in \widehat{x}$. Thus $\bigcup_{s \in S} \widehat{s} \subseteq \widehat{x}$, which implies that
\begin{equation}
    x = \bigvee \bigcup_{s \in S} \widehat{s} \leq \bigvee \widehat{x}.
\end{equation}
Since $\bigvee \widehat{x} \leq x$, we have shown that $x = \bigvee \widehat{x}$. This finishes the induction step.
\end{proof}

\begin{Lemma}[{\cite[Lemma 3.4]{gerhke2013}}] \label{lem Gerhke Van Gool hard lemma}
Given a finite lattice $L$, with $z \in L$ and $A \subseteq L$ a distributive ideal, then $z \in A$ if and only if $\widehat{z} \subseteq \bigcup_{a \in A} \widehat{a}$.
\end{Lemma}

\begin{proof}
If $z \in A$, then clearly $\widehat{z} \subseteq \bigcup_{a \in A} \widehat{a}$. Thus suppose the converse is true. Then since $\bigvee z \meet A \leq z$ it follows that
\begin{equation}
    \widehat{z} = \widehat{z} \cap \bigcup_{a \in A} \widehat{a} = \bigcup_{a \in A} (\widehat{z} \cap \widehat{a}) = \bigcup_{a \in A} \widehat{z \meet a} = \bigcup_{x \in (z \meet A)} \widehat{x} \subseteq \widehat{\left( \bigvee z \meet A \right)} \subseteq \widehat{z}.
\end{equation}
Therefore the two inclusions above are actually equalities. Thus if we let $M = z \meet A$, we have $\bigcup_{x \in M} \widehat{x} = \widehat{ \bigvee M}$. By Corollary \ref{cor equiv def of distr property} this implies that $M =z \meet A$ is a distributive subset of $L$. But $M \subseteq A$, and $A$ is a distributive ideal, so $\bigvee M \in A$. By Lemma \ref{lem Ray's lemma}, we have
\begin{equation}
    z = \bigvee \widehat{z} = \bigvee \widehat{ \left( \bigvee M \right) } = \bigvee M.
\end{equation}
Therefore $z \in A$.
\end{proof}
 
\begin{Prop}[{\cite[Proposition 3.6]{gerhke2013}}]\label{prop bruns lakser as downsets of join irreducibles}
Let $L$ be a finite lattice. Then a subset $X \subseteq L$ is a distributive ideal if and only if it is a downset of join-irreducible elements of $L$. In other words $\cat{BL}(L) = \mathcal{O J}(L)$.
\end{Prop}

\begin{proof}
Suppose that $X \in \mathcal{O J}(L)$. Then every $x \in X$ is join-irreducible, so $X \subseteq \bigcup_{x \in X} \widehat{x}$. But $\widehat{x} \subseteq X$ for every $x \in X$ since $X$ is a downset of join-irreducibles, thus $\bigcup_{x \in X} \widehat{x} = X$. We wish to show that $X$ is a distributive ideal. Suppose that $A \subseteq X$ is distributive. By Corollary \ref{cor equiv def of distr property}, this implies $\widehat{\bigvee A} = \bigcup_{a \in A} \widehat{a}$. But $\bigcup_{a \in A} \widehat{a} \subseteq \bigcup_{x \in X} \widehat{x} = X$. Thus $\widehat{\bigvee A} \subseteq \bigcup_{x \in X} \widehat{x}$. By Lemma \ref{lem Gerhke Van Gool hard lemma} this implies that $\bigvee A \in X$, so $X$ is a distributive ideal. Thus $\mathcal{O J}(L) \subseteq \cat{BL}(L)$.

Now suppose that $X \in \cat{BL}(L)$. Then $\widetilde{X} = \bigcup_{x \in X} \widehat{x}$ is a downset of join-irreducibles, and by the above argument $\tilde{X}$ is a distributive ideal. But note that 
\begin{equation}
    \bigcup_{y \in \widetilde{X}} \widehat{y} = \bigcup_{x \in X} \widehat{x} = \widetilde{X}.
\end{equation}
Therefore for every $y \in \widetilde{X}$, $\widehat{y} \subseteq \bigcup_{x \in X} \widehat{x}$, so by Lemma \ref{lem Gerhke Van Gool hard lemma}, $y \in X$. Thus $\widetilde{X} \subseteq X$. Clearly $X \subseteq \widetilde{X}$, and therefore $X = \widetilde{X}$, so $X$ is a downset of join-irreducibles. So $\cat{BL}(L) = \mathcal{O J}(L)$.
\end{proof}

\begin{Cor}
If $(E, \dep)$ is a DSNC, then $\BL(\rdp(E)) \cong \rdp(E)$.
\end{Cor}

\begin{proof}
If $(E, \dep)$ is a DSNC, then by Corollary \ref{cor dsncs give dist lattices}, $\rdp(E)$ is a finite distributive lattice. Thus by Theorem \ref{th birkhoffs theorem}, $b_{\rdp(E)} : \rdp(E) \to \BL(\rdp(E))$ is an isomorphism.
\end{proof}

With Proposition \ref{prop bruns lakser as downsets of join irreducibles}, we obtain the following characterization of $\cat{BL}(\rdp(E))$ for a DSC $(E, \dep)$. By 
Proposition \ref{prop uniqueness of join irreducible rep}, a complete subset $X \subseteq E$ is join-irreducible in $\rdp(E)$ if and only if it is of the form $D^e \cup e$ for a unique package $e$ and depset $D^e$. This leads directly to the following result.

\begin{Lemma} \label{lem interpretation of J(rdp(E))}
Given a DSC $(E, \dep)$, the poset $\mathcal{J}(\rdp(E))$ is isomorphic to the poset whose elements are pairs $(e, D^e)$ of packages equipped with a choice of depset for that package, where $(e, D^e) \subseteq (e', D^{e'})$ if $D^e \cup e \subseteq D^{e'} \cup e'$. 
\end{Lemma}

By Lemma \ref{lem interpretation of J(rdp(E))} and Proposition \ref{prop bruns lakser as downsets of join irreducibles}, we can therefore interpret an element $X$ of $\mathcal{O J}(\rdp(E))$ to be a set of pairs $(e, D^e)$, where if $(e, D^e) \in X$ and $e' \in D^e$, then there exists some depset $D^{e'}$ such that $(e', D^{e'}) \in X$. We call elements of $\cat{BL}(\rdp(E))$ \textbf{installation sets}. This construction can be made functorial by considering the composite functor
\begin{equation}
    \cat{DSC}_{\text{dis}} \xrightarrow{\rdp} \cat{FinLatt}_{\text{dis}}^\op \xrightarrow{\BL} \cat{FinDLatt}.
\end{equation}

\begin{Ex}
Let us consider what $\cat{BL}(L)$ looks like when $L = \rdp(a.b\join c)$ from Example \ref{ex DSC a depends on b or c}. The join irreducibles of $\rdp(a.b \join c)$ are $\mathcal{J}(L) = \{ b,c,ab,ac \}$. Thus $\mathcal{J}(L)$ looks like
\[\begin{tikzcd}
	ab && ac \\
	b && c
	\arrow[no head, from=1-1, to=2-1]
	\arrow[no head, from=1-3, to=2-3]
\end{tikzcd}\]
Now if we represent the above elements via the isomorphism of Lemma \ref{lem interpretation of J(rdp(E))}, but using the following shorthand
\begin{equation}
b_\varnothing = (b, \varnothing), \qquad c_\varnothing = (c, \varnothing), \qquad (a, \{b \}) = a_b, \qquad (a, \{c \}) = a_c,
\end{equation}
then $\BL(L)$ looks like
\begin{equation}
    \begin{tikzcd}
	& {\{b_\varnothing, c_\varnothing, a_b, a_c\}} \\
	{\{b_\varnothing, c_{\varnothing}, a_b\}} && {\{b_\varnothing, c_\varnothing, a_c\}} \\
	{\{b_\varnothing, a_b\}} & {\{b_\varnothing, c_\varnothing\}} & {\{c_\varnothing, a_c\}} \\
	{\{b_\varnothing\}} && {\{c_\varnothing\}} \\
	& \varnothing
	\arrow[no head, from=2-1, to=1-2]
	\arrow[no head, from=2-3, to=1-2]
	\arrow[no head, from=3-1, to=2-1]
	\arrow[no head, from=3-3, to=2-3]
	\arrow[no head, from=4-1, to=3-1]
	\arrow[no head, from=4-3, to=3-3]
	\arrow[no head, from=3-2, to=2-3]
	\arrow[no head, from=3-2, to=2-1]
	\arrow[no head, from=4-1, to=3-2]
	\arrow[no head, from=4-3, to=3-2]
	\arrow[no head, from=4-1, to=5-2]
	\arrow[no head, from=5-2, to=4-3]
\end{tikzcd}
\end{equation}
The embedding $b: L \to \BL(L)$ is therefore given by
\begin{equation}
\varnothing \mapsto \varnothing, b \mapsto \{b_\varnothing\}, c \mapsto \{c_\varnothing\}, ab \mapsto \{b_\varnothing, a_b \}, bc \mapsto \{b_\varnothing, c_\varnothing \}, ac \mapsto \{c_\varnothing, a_c \}, abc \mapsto \{ b_\varnothing, c_\varnothing, a_b, a_c \}.
\end{equation}
\end{Ex}

The distributive lattice $\BL(\rdp(E))$ has an interesting interpretation as a kind of Merkle tree. Merkle trees were first introduced as a mechanism for digital signatures \cite{merkle1987digital} but now are used widely from storage systems to source control to the hash structure of packages in the Nix package management system. The main idea of the last of these is that it stores ``software components in isolation from each other in a central component store, under path names that contain cryptographic hashes of all inputs involved in building the component'' \cite{dolstra2006purely}. This establishes, effectively, a tree, where each node can be thought of depending on all nodes below it, and is labeled by its own data combined with (the recursive hash of) the subtree of all nodes below it. The act of distinguishing multiple copies of ``the same" event by the record of the choices of their possible inputs corresponds to the construction of a component store in Nix.

Thus $\BL(\rdp(E))$ can be seen as ``splitting" the complete subsets of $E$ whose packages have more than 1 depset. In fact this can already be seen at the level of $\mathcal{J}(\rdp(E))$. By applying $R^{-1}$ from Proposition \ref{prop finpos and dsnc iso}, we obtain a DSNC $\tau(E) = R^{-1}(\mathcal{J}(\rdp(E)))$, which we can think of as a universal way of turning a DSC into a DSNC. Taking Bruns-Lakser of $\rdp(E)$ in turn extends this ``splitting" from packages to complete subsets.

\section{Version Parametrization} \label{section version parametrization}

One motivation for the study of package dependency systems has been to address package version solving and constraint resolution. In fact, different versions of the same package are the prototypical example of dependency sets. Given a desired set of packages to install, picking a uniform set of versions for all their dependencies to enable their simultaneous installation is a nontrivial problem -- in fact, it is NP-complete \cite{dicosmo}.

The difficulty of this problem is often not apparent to end-users of software package systems, because in many cases that work has already been done for them. The repository itself does not consist of all version of packages -- but a single fixed version of each package, such that all are known to be mutually consistent and compatible. The issue presents itself to maintainers of such distribution repositories, but also, more broadly to developers who make use of software library repositories, and contribute their own packages to them as well.

Semi-formal versioning policies have been introduced as a means to aid package maintainers in constructing and managing correct version-bounds. The most prominent and notable of these is Semantic Versioning, (or semver) \cite{semver}. Under semver, the first (major) component of a package version is taken to be for incompatible api changes, and the second (minor) for additions of functionality that do not break backwards compatibility. As semver is a cross-language specification, it does not provide detailed per-language rules on which changes specifically are taken to break backwards-compatibility. The Haskell language also produced its own roughly contemporaneous package versioning policy \cite{pvp}. Because this policy is specific to a single language, it does attempt to give specific rules on what constitute ``breaking" or``non-breaking" changes.

Here we put forward a possible mathematical formalization of update policies, by formalizing the semver notion of a minor or non-breaking update using the following notion of higher version relation on a DSC.

\begin{Def} \label{def higher version relation}
Let $(E, \dep)$ be a DSC. We introduce a relation on $E$, written $e \bhd e'$ if 
\begin{enumerate}
    \item for all $D^e \in \dep(e)$, there exists a depset $D^{e'} \subseteq D^e$, and
    \item for all $x \in E$, if $D^x$ is a depset such that $e \in D^x$, then $[(D^x \setminus e) \cup e']$ is also a depset of $x$. 
\end{enumerate}
We call $\bhd$ the \textbf{higher version relation}. We say that $e'$ is a higher version of $e$.
\end{Def}

\begin{Rem}
One might interpret the above definition as follows. Condition (1) says loosely that higher versions of packages cannot ``gain" new dependencies, and condition (2) says that if a package $x$ can depend on a package $e$ and $e \bhd e'$, then $x$ can depend on $e'$. 
\end{Rem}

\begin{Lemma}
The higher version relation $\bhd$ is reflexive and transitive.
\end{Lemma}

\begin{proof}
It is clear that $\bhd$ is reflexive. Now suppose that $e \bhd e'$ and $e' \bhd e''$. If $e \in D^x$ for some $x \in E$ then $D^x_0 = [(D^x \setminus e) \cup  e'] \in \dep(x)$. But since $e' \bhd e''$, and $e' \in D^x_0$, then $[(D^x_0 \setminus e') \cup e''] = [(D^x \setminus e) \cup e''] \in \dep(x)$. Thus $e \bhd e''$.
\end{proof}

Now if $(E, \dep)$ is a DSC, then consider the following function, $V: P(E) \to P(E)$ defined by
\begin{equation}
 V(X) = \bigcup_{x \in X} \text{Vers}(x).   
\end{equation}

\begin{Lemma}
If $X \subseteq E$ is a complete subset, then $V(X)$ is a complete subset.
\end{Lemma}

\begin{proof}
Suppose that $e' \in V(X)$. Then there exists some $e \in X$ such that $e \bhd e'$. Since $X$ is complete, there exists a depset $D^e \subseteq X$. But $e \bhd e'$ implies that there exists a depset $D^{e'} \subseteq D^e$, and since $X \subseteq V(X)$, $D^{e'} \subseteq V(X)$.
\end{proof}

\begin{Def}
Let $P$ be a poset and $f: P \to P$ a function. We say that $f$ is a \textbf{closure operator} if:
\begin{enumerate}
    \item $f$ is order preserving,
    \item if $X \in P$, then $X \leq f(X)$, and
    \item if $X \in P$, then $f(f(X)) = f(X)$.
\end{enumerate}
Note that this is precisely the same thing as an \textbf{idempotent monad} on $P$ thought of as a category.
\end{Def}

\begin{Cor}
Given a DSC $(E, \dep)$, $V$ descends to a function $V: \rdp(E) \to \rdp(E)$, and is a closure operator.
\end{Cor}

Now we wish to describe how $V$ descends to a closure operator on the Bruns-Lakser completion. If $V: L \to L$ is a closure operator on a finite lattice, we want to show that it induces a closure operator $\bold{V}: \cat{BL}(L) \to \cat{BL}(L)$. If $X \in \cat{BL}(L) = \mathcal{O J}(L)$, then by the proof of Proposition \ref{prop bruns lakser as downsets of join irreducibles}, $X = \bigcup_{x \in X} \widehat{x}$. Thus define
$$\bold{V}(X) = \bigcup_{x \in X} \widehat{V(x)}.$$

\begin{Prop} \label{prop closure operator descends to bruns-lakser}
The function $\bold{V}: \cat{BL}(L) \to \cat{BL}(L)$ as defined above is a closure operator.
\end{Prop}

\begin{proof}
Suppose that $X \subseteq Y$. Then clearly $\bold{V}(X) \subseteq \bold{V}(Y)$. Thus $\bold{V}$ is order preserving.

We wish to show that $X \subseteq \bold{V}(X)$. Suppose that $x \in S$. Then $x \leq V(x)$, since $V$ is a closure operator on $L$. Thus $\widehat{x} \subseteq \widehat{V(x)}$. Therefore $X \subseteq \bigcup_{x \in X} \widehat{V(x)} = \bold{V}(X)$.

Now we wish to show that $\bold{V}(\bold{V}(X)) \subseteq \bold{V}(X)$. Suppose that $z \in \bold{V}(\bold{V}(X))$. Then $z \in \widehat{V(y)}$ for some $y \in \bold{V}(X)$. But then $y \in \widehat{V(x)}$ for some $x \in X$. This implies $y \leq V(x)$, so $V(y) \leq V(V(x)) \leq V(x)$, since $V$ is a closure operator. Thus $\widehat{V(y)} \subseteq \widehat{V(x)}$, and $z \in \widehat{V(x)} \subseteq \bold{V}(X)$, therefore $\bold{V}$ defines a closure operator.
\end{proof}

Thanks to Proposition \ref{prop closure operator descends to bruns-lakser}, the closure operator
$\bold{V}: \cat{BL}(\rdp(E)) \to \cat{BL}(\rdp(E))$ is now defined for any DSC $(E, \dep)$. We call this the \textbf{version monad} on $\cat{BL}(\rdp(E))$.

Frames correspond to complete Heyting algebras, which give in turn models of propositional intuitionistic logic \cite{vickers, maclane}. Furthermore, closure operators yield Kripke frames on these corresponding logics \cite{palmgren}. Consequently, we can associate to any DSC with a choice of versioning relation a corresponding constructive logic with a ``possibility"-like modal operator, and with terms in this logic corresponding to specifications of sets of packages. Our hope is that this ``internal logic of event dependencies" (or related logics) can be a setting which is suited to formal reasoning about and manipulation of propositions regarding event dependencies.

\section{Conclusion} \label{section conclusion}
From a mathematical standpoint, this paper has made the following contributions. By introducing DSCs and Theorems \ref{th equiv DSCs antimatroids} and \ref{th equiv DSC and DES}, we have provided a connection between the theory of general event structures and that of antimatroids, and shown that both can be used as models of package management systems. We have also proposed natural definitions of morphisms of antimatroids and dependency event structures, and explored properties of the resulting equivalent categories. Furthermore, we have given and shown applications of a simple characterization of the finite Bruns-Lakser completion.

From an applied standpoint, we have put this mathematics to work in providing a formal characterization of the dependency data of a package repository, as well as a formal characterization of ``minor" versions of packages, and of package versioning policies. One place we hope this is of use is the field of ``empirical software engineering" and study of ``repository mining". Work there has put significant effort into understanding the structure of software repositories, but at the cost of simplifying assumptions. For example, \cite{riivo} creates dependency graphs with different versions entirely independent, and with each package pinned to a specific single choice of version for each of its dependencies. The work in this paper provides a mathematical structure that can accurately represent both multiple versions of packages with choices of dependencies and also the versioning relation between them.

We have also given a mathematical account of the logical and order-theoretic interpretation of Merkle trees, which sheds light on the reasons for their widespread utility. It is worth noting that the construction and characterization of Bruns-Lakser given in Proposition \ref{prop bruns lakser as downsets of join irreducibles} in fact extends beyond finite lattices to all finite posets, as explored in \cite{bazerman2020topological}.

There remain some interesting outstanding questions regarding morphisms of DSCs. For instance, if $f: (E, \dep) \to (E', \dep')$ is a bimorphism of DSCs, then by Corollary \ref{cor bimorphisms give maps of lattices under rdp}, $f^*$ is a map of lattices. Can we characterize exactly the class of maps of DSCs such that $f^*$ is a map of lattices? Given the connection between DSCs and event structures, it is worth exploring what a "comorphism-first" view of the categorical structure of DSCs would look like. In other words, how do $\cat{DSC}$ and $\cat{DSC}_\co$ differ, and how does using a covariant version of $\rdp$, namely having $f_*: \rdp(E) \to \rdp(E')$ compare with its contravariant version? Are there certain applications to package management systems where comorphisms are more appropriate?

There are a number of possible directions for future work. One is further development of a category whose internal logic can encode package dependency problems. Another is modeling updates to package repositories by means of morphisms in the category of DSCs. Still another is extending DSCs with a notion of conflict, so that they would more fully correspond to general event structures. Beyond that is development of insights into efficient dependency solvers -- a topic of great practical interest. Finally, there is much we hope that can be extracted from the surprising connection between dependency systems and antimatroids. There are many high-powered tools available for the study of the latter, as well as their cousins, matroids, and the more general class of greedoids to which both belong, such as characteristic polynomials. It would be of interest to see what insights these tools can bring when transported to the world of dependency structures.

\printbibliography

\end{document}